\documentclass[11pt]{article}       
\usepackage{amssymb,amsmath,amsthm,latexsym,verbatim,url}
\usepackage{graphicx}
\usepackage[ left=1in, right=1in]{geometry}
\usepackage{natbib}
 \bibpunct[, ]{[}{]}{,}{n}{}{,}%

\usepackage[colorlinks=true,breaklinks=true,bookmarks=true,urlcolor=blue,
     citecolor=blue,linkcolor=blue,bookmarksopen=false,draft=false]{hyperref}

\usepackage{enumerate}

\usepackage{makeidx} 
\usepackage{wrapfig}
\usepackage{amssymb}
\usepackage[TABBOTCAP,tight]{subfigure}

\usepackage{graphicx}
\usepackage[colorlinks=true,citecolor=blue]{hyperref}
\usepackage{pdfsync}
\usepackage{verbatim}

\usepackage{amsmath} 

\newcommand{\Real}{\mathbb{R}}                              
\newcommand{\set}[1]{\left\{#1\right\}}                     
\newcommand{\abs}[1]{\left|#1\right|}                       
\newcommand{\bra}[1]{\left(#1\right)}                       
\newcommand{\eqq}{:=}                                       
\newcommand {\beq}{\begin{equation}}
\newcommand {\eeq}{\end{equation}}
\newcommand {\beqn}{\begin{equation*}}
\newcommand {\eeqn}{\end{equation*}}
\newcommand {\bear}{\begin{eqnarray}}
\newcommand {\eear}{\end{eqnarray}}
\newcommand {\bearn}{\begin{eqnarray*}}
\newcommand {\eearn}{\end{eqnarray*}}

\newcommand{\norm}[1]{\left\lVert#1\right\rVert}

\newcommand{\Z}{\mathbb{Z}}

\newcommand{\Realp}{\mathbb{R}_{+}}
\newcommand {\Realn}{\mathbb{R}^n}
\newcommand {\Realnone}{\mathbb{R}^{n+1}}

\DeclareMathOperator{\conv}{conv}
\DeclareMathOperator{\aff}{aff}
\DeclareMathOperator{\bd}{bd}
\DeclareMathOperator{\cconv}{\overline{conv}}
\DeclareMathOperator{\Int}{int}
\DeclareMathOperator{\lin}{lin}
\DeclareMathOperator{\epi}{epi}
\DeclareMathOperator{\hyp}{hyp}
\DeclareMathOperator{\gr}{gr}

\DeclareMathOperator{\Max}{max}
\newcommand {\phat}{\hat{\pi}}
\newcommand {\xbar}{\overline{x}}

\newcommand {\xstar}{x^*}
\newcommand {\zstar}{z^*}
\newcommand {\ystar}{y^*}
\newcommand {\tstar}{t^*}
\newcommand {\tbar}{\overline{t}}
\newcommand{\norms}[1]{\left\|#1\right\|_2}
\newcommand{\normss}[1]{\left\|#1\right\|_2^2}

\newcommand{\mybar}[1]{\overline{#1}}

\newcommand {\bcom}{}

\usepackage{thmtools,thm-restate}

\newtheorem{corollary}{Corollary}[section]
\newtheorem{lemma}{Lemma}[section]
\theoremstyle{definition}
\newtheorem{example}{Example}[section]

\newtheorem{definition}{Definition}[section]
\newtheorem{proposition}{Proposition}[section]

\usepackage{lineno}
\usepackage{lipsum}

\begin{document}

\title{Intersection cuts for nonlinear integer programming: convexification techniques for structured sets
}


\author{Sina Modaresi\thanks{Department of Industrial Engineering, University of Pittsburgh, Pittsburgh, PA 15261, {\tt sim23@pitt.edu}}         \and
        Mustafa R. K\i l\i n\c{c}\thanks{Department of Chemical Engineering, Carnegie Mellon University, Pittsburgh, PA 15213, {\tt mkilinc@andrew.cmu.edu}} \and
        Juan Pablo Vielma \thanks{Sloan School of Management, Massachusetts Institute of Technology, Cambridge, MA 02139, {\tt jvielma@mit.edu}}
}



\maketitle

\begin{abstract}
We study the generalization of split, k-branch split, and intersection cuts from Mixed Integer Linear Programming to the realm of Mixed Integer Nonlinear Programming. 
Constructing such cuts requires calculating the convex hull of the difference between a convex set and an open set with a simple geometric structure. We introduce
two techniques to give precise characterizations of such convex hulls and use them to construct split, k-branch split, and intersection cuts for several classes of non-polyhedral sets. 
In particular, we give simple formulas for split cuts for essentially all convex sets described by a single quadratic inequality. We also give simple formulas for k-branch split cuts and some general intersection cuts for a wide variety of convex quadratic sets.
\end{abstract}


\section{Introduction}\label{intro}


An important area of Mixed Integer Linear Programming (MILP) is the characterization of the convex hull of specially structured non-convex polyhedral sets to develop strong valid inequalities or cutting planes such as split and intersection cuts  \cite{conforti2010polyhedral,conforti2011corner,cornuejols2008valid,del2012relaxations}. This approach has led to highly effective branch-and-cut algorithms  \cite{DBLP:journals/mpc/Achterberg09,DBLP:journals/anor/BixbyR07,bixby2004,DBLP:journals/informs/JohnsonNS00,Lodi2009}, so there has  recently been significant interest in extending the associated theoretical and computational results to the realm of Mixed Integer Nonlinear Programming (MINLP) \cite{Atamturk07,AN:conicmir,Goez2012,dan,DBLP:conf/ipco/Bonami11,DBLP:journals/mp/CezikI05,DBLP:journals/mor/DadushDV11,DBLP:conf/ipco/DadushDV11,Dadush2011121,Drewes,mustafa,springerlink:10.1007/s101070050103}. Unfortunately, this extension requires the study of the convex hull of a non-convex and non-polyhedral set, which has proven to be significantly harder than the polyhedral case. Most of the known results in this area are limited to very specific sets \cite{horst2003global,sherali1998reformulation,tawarmalani2002convexification} or to approximations of semi-algebraic sets through Semidefinite Programming (SDP) \cite{fujie1997semidefinite,lasserre2001global,nesterov,oustry2001sdp,parrilo2003semidefinite,polik2007survey,poljak1995recipe}. While some precise SDP representations of the convex hulls of semi-algebraic sets exist \cite{Gouveia12,Helton12,henrion2011semidefinite,Scheiderer20112606}, these require the use of auxiliary variables. Such higher dimensional, extended, or lifted representations are extremely powerful. However, there are theoretical and computational reasons to want representations in the original space and/or in the same class as the original set (e.g. representations that do not jump from quadratic basic semi-algebraic to SDP). We refer to characterizations that satisfy both these requirements as \emph{projected} and  \emph{class preserving}. Projected and class preserving are in general incompatible (e.g. the convex hull of the basic semi-algebraic set $\set{x\in \Real^2\,:\, (x_1^2-x_2)x_1\geq 0,\, x_2\geq 0}$ has no projected basic semi-algebraic representation, but has a lifted basic semi-algebraic representation \cite{lecnote}). Furthermore, even giving an algebraic characterization of the boundary of the convex hull of a variety \cite{ranestad2011convex,ranestad2009convex} or giving a projected SDP representation of the convex hull of certain varieties and quadratic semi-algebraic sets \cite{sanyal2011orbitopes,yildiran2009convex,Kose}  requires very complex techniques from algebraic geometry. All such issues make extending MILP cutting planes to the MINLP setting extremely challenging. To alleviate such challenges, we concentrate on the extension of split cuts, k-branch split cuts, and other intersection cuts to the MINLP setting \cite{balas1971intersection,DBLP:journals/mp/CookKS90,dash2012two,Gomory1969451,springerlink:10.1007/BF01584976,li2008cook}. 

Split, k-branch split, and intersection cuts for MILP can all be obtained by taking the convex hull of the difference between a convex set and a  set with a simple geometric structure.  This characterization allows for a straightforward extension of the cuts to the MINLP setting. However, this conceptual extension does not provide a practical construction procedure for the cuts. For this reason, we follow the approach of  the simple, but extremely powerful Mixed Integer Rounding (MIR) cut \cite{Marchand,DBLP:books/daglib/0090563,springerlink:10.1007/BF01585752,Wolsey}. The MIR procedure can be used to generate every split cut for a MILP and, together with the closely related Gomory Mixed Integer (GMI) cut procedure \cite{Gomory1969451,springerlink:10.1007/BF01584976}, yields the most effective cutting plane approach for general MILP \cite{DBLP:journals/anor/BixbyR07,bixby2004}. In particular, one version of the MIR procedure shows that every split cut can be constructed through a simple two step procedure. The first step is the construction of a canonical cut known as the \emph{simple} or \emph{basic} MIR. This cut is obtained by taking the convex hull of the difference between two simple convex sets in $\Real^2$, both of which are described by two linear inequalities. The second step simply uses linear transformations to obtain all split cuts from the basic MIR.  In this paper we show that a similar approach can be used to construct a wide range of intersection cuts. More specifically, we show how two very simple techniques can be used to construct projected class preserving characterizations of the convex hull of difference between certain canonical sets. The techniques we consider are only tailored to the geometric structure of these canonical sets  and do not require the sets to have any additional  algebraic properties (e.g. being quadratic, basic semi-algebraic, etc.). Thanks to this, the resulting characterizations are quite general, but give simple closed form expressions. While the canonical sets are somewhat specific, we can also use affine transformations to obtain more general cuts. In particular, these techniques can be used to construct split cuts for essentially all convex sets described by a single conic quadratic inequality, and to extend k-branch split and general intersection cuts to a wide variety of quadratic sets of interest to trust region and lattice problems. In both cases, the only algebraic property of the quadratic sets needed for the construction is the symmetry of the Euclidean norm. This suggests that the techniques could be useful to construct cuts for additional classes of sets by only exploiting similar basic properties.

The rest of this paper is organized as follows. We begin with Section~\ref{LitReview} where we introduce some notation and review some known results. Section~\ref{NSCSec} then introduces an interpolation technique that can be used to construct split and k-branch split cuts for many classes of sets. Then, in Section~\ref{intquadsec} we use the interpolation technique to characterize intersection cuts for conic quadratic sets. Finally,  Section~\ref{intersection_cut_section}  introduces an aggregation technique that can be used to construct a wide array of general intersection cuts. In both Sections~\ref{NSCSec} and \ref{intersection_cut_section}, we first present the basic principles behind the techniques in a simple, but abstract setting, and then utilize them to construct more specific cuts to illustrate their power and limitations.  

\section{Notation, known results and other preliminaries} \label{LitReview}

We use the following  notation. Let $e^i\in\mathbb{R}^n$ be the $i$-th unit vector, $0_n \in \Real^n$ be the zero vector, and $I\in \mathbb{R}^{n\times n}$ be   the identity matrix where $n$ is an appropriate dimension that we omit if evident from the context.  We also let $\norm{x}_2:=\sqrt{\sum_{i=1}^n x_i^2}$ denote the Euclidean norm of a given vector $x\in \mathbb{R}^n$ and for a vector $v \in \Real^n$, we let the projection onto its span be $P_{v} := \frac{v v^T}{\|v \|_2^2}$ and  onto its  orthogonal complement be $P_{v}^ {\perp} := I - \frac{v v^T}{\|v \|_2^2}$. We also let $\set{\pi_i}_{i=1}^k \subseteq \Real^n\setminus\set{0_n}$ be an arbitrary set of vectors, and not necessarily a sequence of vectors. For a set $S\subseteq \Real^n$, we let $\Int\bra{S}$ be its interior, $\bd\bra{S}$ be its boundary, $\conv\bra{S}$ be its convex hull, $\cconv\bra{S}$ be the closure of its convex hull, $\aff\bra{S}$ be its affine hull, and $\lin\bra{S} \eqq \set{d \in \Realn \,:\, x + \lambda d \in S \mbox{~for all~} x \in S \mbox{~and~} \lambda \in \Real}$ be its lineality space. For a function $G:\mathbb{R}^n\to \mathbb{R}$  we let $\epi\bra{G}:=\set{\bra{x,t}\in \Realnone\,:\, G(x)\leq t}$ be its epigraph, $\gr\bra{G}:=\set{\bra{x,t}\in \Realnone\,:\, G(x)= t}$ be its graph, and $\hyp\bra{G}:=\set{\bra{x,t}\in \Realnone\,:\, G(x)\geq t}$ be its hypograph. In addition, we let $[n]:=\set{1,\ldots,n}$.

\begin{definition}[Intersection, Split, k-branch Split, and t-inclusive Split Cuts] \label{firstdef} Let $B\subseteq\mathbb{R}^{n}$ be a closed convex set that we refer to as the \emph{base set}, $F\subseteq\mathbb{R}^{n}$ be a closed set that we refer to as the \emph{forbidden set}, and $g:\Real^{n}\to\Real$ be an arbitrary function.
We say inequality $g(x)\leq 0$ is an \emph{intersection cut} for $B$ and $F$ if  $\cconv\bra{B\setminus \Int\bra{F}}\subseteq \set{x\in \Real^{n}\,:\, g(x)\leq 0}$ and   $g$ is convex.

	We let a  \emph{split} be a set of the form $\set{x\in \Real^{n}\,:\, \pi^Tx\in[\pi_0,\pi_1]}$ for some $\pi\in \Real^n\setminus\set{0_n}$ and $\pi_0,\,\pi_1\in \Real$  such that $\pi_0<\pi_1$.
	If $F$ is a split, we say that the associated intersection cut is a \emph{split cut}. Besides, if $F$ is a split with $\pi=e^i$ for some $i\in [n]$, we refer to $F$ as an \emph{elementary split} and to the the associated split cut as an \emph{elementary split cut}.
	
	We let a \emph{k-branch split} be a set of the form $\bigcup_{i=1}^k \set{x\in \Real^{n}\,:\, \pi_0^i\leq \pi_i^T x\leq \pi_1^i}$ for some $\set{\pi_i}_{i=1}^k \subseteq \Real^n\setminus\set{0_n}$,  $\pi_0^i, \pi_1^i \in \Real$ such that $\pi_0^i<\pi_1^i$ for all $i\in [k]$. If $F$ is a k-branch split,  we say that the associated intersection cut is a \emph{k-branch split cut}.
	
When considering epigraphical sets of the form $B = \set{\bra{x,t} \in \Realnone\,:\, G\bra{x} \le t}$ for some closed convex function $G\bra{x}$, we often assume that $F$ is a cylinder whose axis lies along $t$ (i.e., $F$ is of the form $S \times \Real$ for some $S \subseteq \Real^n$). For instance, if $F$ is a split, we have $S = \set{(x,t)\in \Real^{n+1}\,:\, \pi^Tx\in[\pi_0,\pi_1]}$. However, in some cases, we consider a split that includes $t$ and we refer to such a split as a \emph{t-inclusive split}. More specifically, we let a t-inclusive split be a set of the form $\set{(x,t)\in \Real^{n+1}\,:\, \pi^Tx + \phat t \in[\pi_0,\pi_1]}$ for some $(\pi,\phat)\in \Real^{n+1}$ such that $\phat \neq 0$\footnote{We allow $\pi=0_n$ to consider disjunctions that only affect $t$.}, and $\pi_0,\,\pi_1\in \Real$  such that $\pi_0<\pi_1$. 
	If $F$ is a t-inclusive split,  we say that the associated intersection cut is a \emph{t-inclusive split cut}.
	
	We mostly restrict to the cases in which $\conv\bra{B\setminus \Int\bra{F}}$ is closed, so for notational convenience,  we let $\mybar{B} \eqq \conv\bra{B\setminus \Int\bra{F}}$ when $F$ is evident from the context.
\end{definition}
We note that the term intersection cut was introduced by Balas \cite{balas1971intersection} for the case in which $B$ is a translated simplicial cone, $F$ is convex and the unique vertex of $B$ is in $\Int\bra{F}$. In this setting, we have that $\conv\bra{B\setminus\Int\bra{F}}$ is closed and can be described by adding a single linear inequality to $B$. Furthermore, this single linear inequality has a simple formula dependent on the  \emph{intersections} of the extreme rays of $B$ with $F$. While we do not always have such intersection formulas for other classes of sets, we continue to use the term intersection cut in the more general setting and avoid any additional qualifiers for simplicity. In particular, we do not use the term \emph{generalized intersection cut} as it has already been used for the case of polyhedral $B$ and $F$ and in conjunction with an improved cut generation procedure for MILP \cite{genintersec}. The term split cut was introduced by Cook, Kannan and Schrijver \cite{DBLP:journals/mp/CookKS90}, and their original definition directly generalizes to non-polyhedral sets as in Definition~\ref{firstdef}. The term k-branch split cut was introduced by Li and Richard \cite{li2008cook}; 2-branch split cuts are also called cross cuts in  Dash, Dey and G{\"u}nl{\"u}k \cite{dash2012two}. These definitions also directly generalize to non-polyhedral sets as in Definition~\ref{firstdef}.

The interest of intersection cuts for MILP and MINLP arises from the fact that if $\Int(F)\cap  \mathbb{Z}^p\times \mathbb{R}^q=\emptyset$, an intersection cut for $B$ and $F$ is valid for  $\cconv\bra{B\cap \mathbb{Z}^p\times \mathbb{R}^q}$.
Hence, intersection cuts can be used to strengthen the continuous relaxation of MILP and MINLP problems.

Intersection cuts are particularly attractive in the MILP setting, since they can be quite strong and can be easily constructed. They were extensively studied when they were first
proposed in the 1970s  \cite{balas1971intersection,Gomory1969451,springerlink:10.1007/BF01584976} and have recently received renewed interest \cite{conforti2011corner,del2012relaxations}. Part of the relative simplicity and effectiveness of intersection cuts for MILP stems from two basic facts. The first one is that in the MILP setting, $B$ is a  polyhedron (i.e., the continuous relaxation of a MILP is an LP). The second one is the fact that every convex set $F$ such that $\Int(F)\cap  \mathbb{Z}^n=\emptyset$ (usually denoted a \emph{lattice free convex set}) and that is maximal with respect to inclusion for this property is also a  polyhedron \cite{lovasz1989geometry}. Restricting both $B$ and $F$ to be (convex) polyhedra give intersection cuts for MILP several useful properties. For instance, if $B$ and $F$ are polyhedra, then $\cconv\bra{B\setminus \Int\bra{F}}$ is a polyhedron   \cite{del2012relaxations}. Hence, in the MILP setting, we can  restrict our attention to linear intersection cuts. Furthermore, if $B$ is a translated simplicial cone and its unique vertex is $\Int\bra{F}$, then $\conv\bra{B\setminus \Int\bra{F}}$ is closed, can be described by adding a single linear inequality to $B$, and this linear inequality has a relatively simple formula  \cite{balas1971intersection,Gomory1969451,springerlink:10.1007/BF01584976}. In particular, if $F$ is a split and $B$ is a polyhedron, then all linear intersection cuts for $B$ and $F$  can be constructed from simplicial relaxations of $B$ and hence have simple formulas  \cite{DBLP:journals/mp/AndersenCL05,dash2011note,DBLP:journals/orl/Vielma07}. As discussed in Section~\ref{intro}, GMI cuts \cite{Gomory1969451,springerlink:10.1007/BF01584976} and MIR cuts  \cite{Marchand,DBLP:books/daglib/0090563,springerlink:10.1007/BF01585752,Wolsey} are two versions of these formulas. For more information on the ongoing efforts to duplicate this effectiveness for other lattice free polyhedra,  we refer the reader to \cite{conforti2011corner,del2012relaxations}. In this context, we note that $\conv\bra{B\setminus \Int\bra{F}}$ can fail to be closed even if $B$ and $F$ are polyhedra and $F$ is not a split (e.g. consider $B=\set{x\in \Real^2\,:\, x_2\geq 0}$ and $F=\set{x\in \Real^2\,:\, x_2\leq 1,\,x_1+x_2\leq 1}$). However,  $\conv\bra{B\setminus \Int\bra{F}}$ is closed in the polyhedral case if $F$ is convex and full-dimensional and the recession cone of $F$ is a linear subspace \cite{andersen2010analysis}. 

In the MINLP setting, there has been significant work on the computational use of linear split cuts  \cite{DBLP:conf/ipco/Bonami11,DBLP:journals/mp/CezikI05,springerlink:10.1007/s101070050103,Drewes,  mustafa}.  From the theoretical side, we know that  if $F$ is a split, then  $\conv\bra{B\setminus \Int\bra{F}}$ is closed even if $B$ is not polyhedral \cite{Dadush2011121}. With respect to formulas for intersection cuts, there has been some progress in the description of split cuts for quadratic sets in \cite{Atamturk07,AN:conicmir,Dadush2011121,Goez2012}. Dadush et al.  \cite{Dadush2011121} show that, if $B$ is an ellipsoid and $F$ is a split, then $\conv\bra{B\setminus \Int\bra{F}}$ can be described by intersecting $B$ with either a linear half space, an affine transformation of the second-order cone  (a.k.a. Lorentz cone), or an ellipsoidal cylinder. In addition, they give simple closed form expressions for all these linear and nonlinear split cuts.   Independently, \cite{Goez2012} studies split cuts for more general quadratic sets, but only for splits in which $\{x\in B\,:\, \pi^Tx= \pi_0\}$ and $\{x\in B\,:\, \pi^Tx= \pi_1\}$  are bounded. They give a procedure to find the associated split cuts, but do not give closed form expressions for them.   Finally, \cite{Atamturk07,AN:conicmir}  give a simple formula for an elementary split cut for the standard three dimensional second-order cone. While  \cite{Goez2012}  develops a procedure to construct split cuts through a detailed algebraic analysis of quadratic constraints developed in  \cite{belotti2013families}, \cite{Atamturk07,AN:conicmir,Dadush2011121} give formulas for split cuts through simple geometric arguments. As we have recently shown at the MIP 2012 Workshop, these geometric techniques can be extended to additional quadratic and basic semi-algebraic sets \cite{poster}. In this paper we show that the principles behind these geometric arguments can be abstracted from the semi-algebraic setting  to develop split and k-branch split cut formulas for  a wider class of specially structured convex sets. This abstraction greatly simplifies the proofs and can be used to construct  split cuts for essentially all convex sets described by a single quadratic inequality through simple linear algebra arguments. In addition to studying split and k-branch split cuts, we show how a commonly used  aggregation technique can be used to develop formulas for general nonlinear intersection cuts for the case in which $B$ and $F$ are both non-polyhedral, but share a common structure. While a non-polyhedral $F$ is not necessary in the MINLP settings (it still should be sufficient to consider maximal lattice free convex sets, which are polyhedral), they could still provide an advantage and are important in other settings such as trust region problems \cite{dan,polik2007survey} and  lattice problems \cite{DBLP:conf/ipco/BuchheimCL10,springerlink:10.1007/s10107-011-0475-x,MGbook}. We finally note that similar results for the quadratic case have  recently been independently developed in \cite{Kent}. We discuss the relation between the results in \cite{Kent} and our work at the end of Section~\ref{complicated}.

To describe our approach, we use the following additional definition. 
\begin{definition} Let $B\subseteq\mathbb{R}^{n}$ be a closed convex set, $F\subseteq\mathbb{R}^{n}$ be a closed set, and $g:\Real^n\to\Real$ be an arbitrary function.
We say inequality $g(x)\leq 0$ is a:
	\begin{itemize}
  \item \emph{valid cut}  if  $\mybar{B}\subseteq \set{x\in \Real^{n}\,:\, g(x)\leq 0}$,
	\item \emph{binding valid cut}  if it is valid and $\set{x\in B\setminus \Int\bra{F}\,:\, g(x)=0}\neq \emptyset$, and 
		\item \emph{sufficient cut}, if $\set{x\in B \,:\, g(x)\leq0}\subseteq \mybar{B}$.
	\end{itemize}
\end{definition}

Binding valid  cuts correspond to valid cuts that cannot be improved by translations, and sufficient cuts are those that are violated by any point of $B$ outside $\mybar{B}$.  We can show that a convex cut that is sufficient and valid  is enough to describe $\mybar{B}$ together with the original constraints defining $B$. Our approach to generating such cuts will be to construct cuts that are binding and valid by design, and that have simple structures from which sufficiency can easily be proven.

\section{Intersection cuts through interpolation} \label{NSCSec}

In this section we consider the case in which the base set is either the epigraph, lower level set, or a section of the epigraph of a convex function and the forbidden set corresponds to a split, t-inclusive split, or a k-branch split. Our cut construction approach is based on a simple interpolation technique that can be more naturally explained for splits and epigraphs of specially structured functions. For this reason, we begin with such a case and then consider special cases of non-epigraphical sets and discuss the limits of the interpolation technique. While the structures for which the technique yields simple formulas are quite specific, we can consider broader classes by considering affine transformations. In Section~\ref{intquadsec} we illustrate the power of this approach by showing how the interpolation technique yields formulas for intersection cuts for convex quadratic sets.

\subsection{Split cuts for epigraphical sets}

Let  $G:\mathbb{R}\to \mathbb{R}$ be a closed convex function, 
\begin{equation}
 \epi(G) := \set{\bra{z, t} \in \Real \times \Real :  G\left(z\right) \le t}
\end{equation}
be its epigraph, and let $F$ be an elementary split associated with $\pi = e^1$. Then  $\mybar{\epi(G)}=\epi(G)\cap \epi(J)$ for  
\begin{equation}\label{simple_splitcut}
	 J(z)=\frac{G(\pi_1) - G(\pi_0)}{\pi_1 - \pi_0} z +\frac{\pi_1 G(\pi_0) - \pi_0 G(\pi_1)}{\pi_1- \pi_0}.
\end{equation}
This is illustrated in Figure~\ref{firstex1}, where the graph of $G$ is given by the thick black curve and the graph of $J$ is depicted by the thin blue line. 
\begin{figure}[htb]
\centering 
\subfigure[Naive friends construction.]{\includegraphics[scale=0.35]{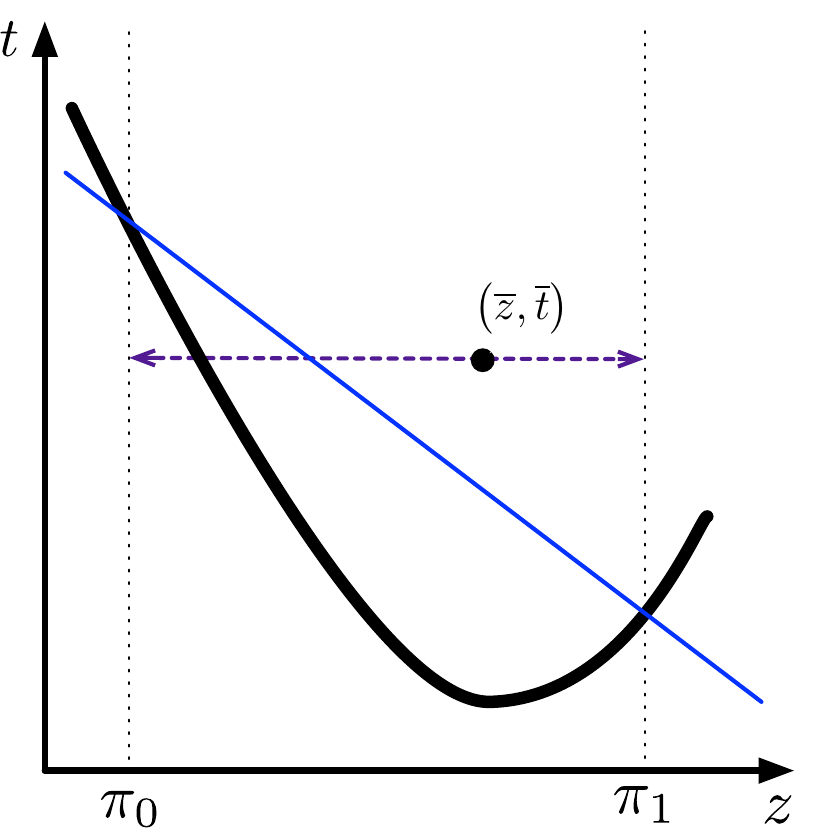}\label{firste2}}
\subfigure[Friends by following the slope.]{\includegraphics[scale=0.35]{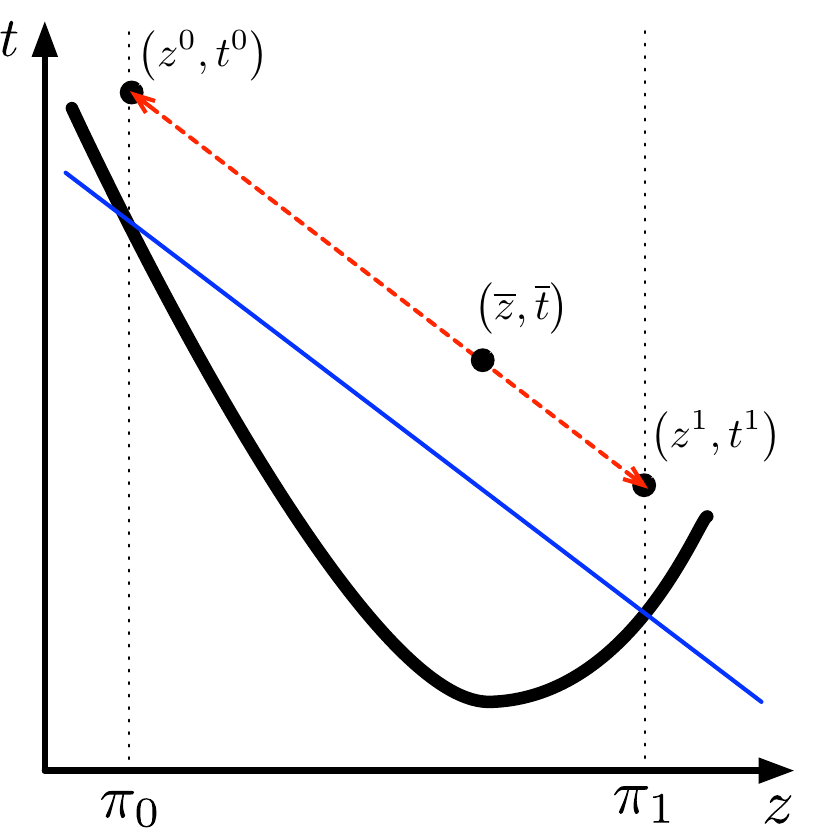}\label{firstex3}}
\caption{Interpolation technique for univariate functions.}\label{firstex1}
\end{figure}
Indeed, since $J$ is a linear function and hence  $\epi(G)\cap \epi(J)$  convex, it is enough to show 
that $J(z)\leq t$ is a valid and sufficient cut. 

We can check that $J(z)\leq t$ is a binding valid cut by design. Indeed,  $J$ is the (affine) linear interpolation of $G$ through $z=\pi_0$ and $z=\pi_1$. Convexity of $G$ then implies that this interpolation is below $G$ outside $z\in(\pi_0,\pi_1)$.

To show that the cut is sufficient, we need to show that any point $\bra{\overline{z},\overline{t}} \in \epi(G)$ that satisfies the cut is in $\mybar{\epi(G)}$. To achieve this, we can find two points $\bra{z^0,t^0}$ and $\bra{z^1,t^1}$ in $\epi(G)$ such that $z^0\leq \pi_0$, $z^1\geq \pi_1$, and $\bra{\overline{z},\overline{t}}\in\conv\bra{\set{\bra{z^0,t^0},\bra{z^1,t^1}} }$. Following \cite{dgv11}, we will denote these points the \emph{friends} of $\bra{\overline{z},\overline{t}}$. One naive way to construct the friends is to wiggle $\bra{\overline{z},\overline{t}}$  by  decreasing and increasing $\overline{z}$ until it reaches $\pi_0$ and $\pi_1$, respectively. However, as illustrated in Figure~\ref{firste2}, this can result in one of the friends falling outside $\epi(G)$. Fortunately, as illustrated in Figure~\ref{firstex3}, we can always wiggle by following the slope of the cut $J$ to assure that the friends are in $\epi(J)$. Correctness (i.e., containment of the friends in $\epi(G)$) then follows by noting that $J(z)=G(z)$  at $z=\pi_0$ and $z=\pi_1$, since $J(z)\leq t$ is a binding valid cut.
This two-stage procedure of binding validity through interpolation and sufficiency through friends can be formalized for general closed convex sets as follows.

\begin{proposition}\label{method1Cross}
	Let $B \subseteq \Real^n$ be a closed convex set and $F \subseteq \Real^n$ be closed. If $C \subseteq \Real^n$ is a closed convex set such that 
	\begin{subequations}\label{generalintcondCross}
		\bear
			B \cap \bd\bra{F} &=& C \cap \bd\bra{F} \label{generalintcondbCross}\\
			B \setminus \Int\bra{F} &\subseteq & C \setminus \Int\bra{F},\label{generalintcondcCross}
		\eear
	\end{subequations}
	and if
	\beq
	\mbox{for all~} \xbar \in C \cap \Int\bra{F} \mbox{there exists a finite set } \Gamma \subseteq C \cap \bd\bra{F} \mbox{ such that } \xbar\in \conv\bra{\Gamma}, \label{friendcond}
	\eeq
	
  then 
	\begin{equation}
		\mybar{B} = B \cap C.
	\end{equation}
\end{proposition}

\begin{proof}
We have that
\begin{equation}\label{AAAAACross}
B \setminus \Int\bra{F} \subseteq B \cap C \subseteq \mybar{B},
\end{equation}
where the first containment comes from \eqref{generalintcondcCross} and the last from \eqref{friendcond} and \eqref{generalintcondbCross}. The result follows by taking convex hull in \eqref{AAAAACross} and noting that $B \cap C$ is convex because both $B$ and $C$ are convex. 
\qed\end{proof}

Note that if $F$ is a split, we can always consider $\Gamma$ containing exactly two points (e.g Figure~\ref{firstex1} and Propositions~\ref{freefriendslinprop} and \ref{freefriendsconeprop}), while larger sets $\Gamma$ might be necessary for other forbidden sets (e.g. Proposition~\ref{GFormSplitCutCross}).
Our general approach to use Proposition~\ref{method1Cross} is to construct a convex function that yields binding valid cuts (i.e., satisfies \eqref{generalintcondCross}) and to use its specific geometric structure to construct friends for sufficiency. We now consider two structures in which the appropriate interpolation can easily be constructed once we identify the interpolations general form. The geometric structures of the resulting cuts yield two friends construction techniques. The first technique generalizes the univariate argument in Figure~\ref{firstex3} by noting that following the slope of $J$ is equivalent to moving in $\lin\bra{ \epi(J)}$. The second technique constructs the friends by moving in a ray contained in an appropriately constructed cone. These techniques are described in detail in Sections~\ref{SepFunSplit} and \ref{ENSF Section} respectively.
\subsubsection{Separable functions} \label{SepFunSplit}
Let $G$ be a separable function of the form $G(z,y)=f(z)+g(y)$ with $f:\Real \to \Real$ and $g:\Real^p\to \Real$ closed convex functions, and let $F$ be an elementary split associated with $\pi = e^1$. Analogous to \eqref{simple_splitcut}, we can simply interpolate $G$  parametrically on $y$ to obtain 
\begin{equation}\label{simple_splitcut_para}
	 J(z,y)=\frac{G(\pi_1,y) - G(\pi_0,y)}{\pi_1 - \pi_0} z +\frac{\pi_1 G(\pi_0,y) - \pi_0 G(\pi_1,y)}{\pi_1- \pi_0}.
\end{equation}
In this case, the interpolation simplifies to  
\begin{equation*}
	 J(z,y)= \frac{f(\pi_1) - f(\pi_0)}{\pi_1 - \pi_0} z +\frac{\pi_1 f(\pi_0) - \pi_0 f(\pi_1)}{\pi_1- \pi_0}+g(y),
\end{equation*}
which is convex on $\bra{z,y}$ and linear on $z$. Our original univariate argument follows through directly and we get $\mybar{\epi\bra{G}} = \epi\bra{G} \cap \epi\bra{J}$.
To illustrate this, consider $G:\Real\times \Real\to \Real$ given by $G(z,y)=z^2+y^2$ and let $F$ be the elementary split associated with $\pi = e^1$, $\pi_0=-10$, and $\pi_1=1$. Constructing a parametric linear interpolation as in \eqref{simple_splitcut_para} yields  
 \begin{equation*}
	 J(z,y)=\frac{1-100}{11}z + \frac{\bra{100 + y^2} + 10 \bra{1 + y^2}}{11}= -9z + 10 + y^2.
 \end{equation*}
 
 Function $J$ is convex on $(z,y)$,  linear on $z$, and can be easily shown to satisfy the conditions of Proposition~\ref{method1Cross}. We can thus conclude that it yields the associated split cut.
 In contrast, if we consider the non-elementary split $\pi = \bra{1,1}^T$ with the previous choices of $\pi_0$ and $\pi_1$ on the same function $G$, we need to proceed with more care. In particular, the parametric interpolation \eqref{simple_splitcut_para} cannot be directly applied since the disjunction affects both $z$ and $y$. However, we can construct the split cut by exploiting the fact that $G$ can be represented as 
\begin{equation}\label{piperpeq}
 G(z,y)= \frac{\bra{z+y}^2}{2} + \frac{\bra{z-y}^2}{2}= \frac{\bra{\pi^T (z,y)}^2}{2}+\frac{\bra{h^T(z,y)}^2}{2},
\end{equation}
where $h=(1,-1)^T$  is orthogonal to $\pi$. If we let $\tilde{z} = \pi^T (z,y)$,   $\tilde{y}=h^T(z,y)$, $\tilde{\pi}=\bra{1,0}$, $\tilde{\pi}_0=-10$, $\tilde{\pi}_1=1$, and $ \tilde{G}\bra{\tilde{z},\tilde{y}}=\tilde{z}^2/2+\tilde{y}^2/2$, we revert to the elementary case where we can apply the parametric interpolation \eqref{simple_splitcut_para} to obtain the split cut 
\begin{equation}
   \tilde{J}\bra{\tilde{z},\tilde{y}}= \frac{\tilde{G}\bra{\tilde{\pi}_1,\tilde{y}} - \tilde{G}\bra{\tilde{\pi_0},\tilde{y}}}{\tilde{\pi}_1 - \tilde{\pi}_0} \tilde{z} +\frac{\tilde{\pi}_1 \tilde{G}\bra{\tilde{\pi}_0,\tilde{y}} - \tilde{\pi}_0 \tilde{G}\bra{\tilde{\pi}_1,\tilde{y}}}{\tilde{\pi}_1- \tilde{\pi}_0}=\frac{-9\tilde{z} + 10 + \tilde{y}^2}{2}.
 \end{equation}
 We can then recover the split cut in the original $(z,y)$ space by replacing the definitions of $\tilde{z}$ and $\tilde{y}$.
The same procedure can be used for any separable function that is of, or can be converted to, the form $G(x)=  f\left(\pi^T x\right) + g\left(P_\pi^{\perp} x\right)$ where $g:\mathbb{R}^n\to \mathbb{R}$ and $f:\mathbb{R}\to \mathbb{R}$ are closed convex functions and $P_{\pi}^{\perp}:= I - \frac{\pi \pi^T}{\|\pi \|_2^2}\in \Real^{n\times n}$ is the matrix associated with the projection onto the orthogonal complement of $\pi$ ($P_{\pi}^{\perp}x$ plays the same role as $h^T(z,y)$ in \eqref{piperpeq}).
To formally prove this, we first show how the friends construction procedure of Figure~\ref{firstex3} can be extended to a general closed convex set $C$ by considering  properties of $\lin\bra{C}$. 

\begin{proposition}\label{freefriendslinprop}
Let $F \subseteq \Realn$ be a split and $C \subseteq \Realn$ be a closed convex set. If there exists $u \in \lin\bra{C}$ such that $\pi^T u \neq 0$, then  condition  \eqref{friendcond} in Proposition~\ref{method1Cross} is satisfied.
\end{proposition}

\begin{proof}
Let $\xbar \in C$ such that $\pi^T \xbar \in \bra{\pi_0,\pi_1}$ and $u \in \lin\bra{C}$ such that $\pi^T u \neq 0$. Also let   $x^i \eqq \xbar + \lambda_i u$ for $i \in \set{0,1}$, where 
\beqn
\lambda_i = \frac{\pi_i - \pi^T \xbar}{\pi^T u},
\eeqn
and let  $\beta \in (0,1)$ be such that $\pi^T \xbar = \beta \pi_0 + \bra{1-\beta} \pi_1$.
Because $u \in \lin\bra{C}$ and since $\pi^T x^i = \pi_i$,  we have $x^i \in C \cap \bd\bra{F}$ for $i \in \set{0,1}$.  The results then follows by noting that $\xbar = \beta x^0 + \bra{1-\beta} x^1$.
\qed
\end{proof}

Using Propositions~\ref{method1Cross} and \ref{freefriendslinprop}  we obtain the following split cut formula for separable functions. 

\begin{restatable}{proposition}{GFormSplitCut} \label{GFormSplitCut}
	Let $F$ be a split, $g:\mathbb{R}^n\to \mathbb{R}$ and $f:\mathbb{R}\to \mathbb{R}$ be closed convex functions, \[ S_{g,f} := \set{\bra{x, t} \in \Realnone : g\left(P_\pi^{\perp} x\right) + f\left(\pi^T x\right) \le t},\]
$a = \frac{f(\pi_1) - f(\pi_0)}{\pi_1 - \pi_0}$, and $b = \frac{\pi_1 f(\pi_0) - \pi_0 f(\pi_1)}{\pi_1- \pi_0}$. Then $\mybar{S_{g,f}} = S_{g,f} \cap C$, where
\[C = \set{ \bra{x, t} \in \Realnone\,:\, g\left(P_\pi^{\perp} x\right) + a \pi^T x + b \le t}.\]
\end{restatable}
\begin{proof}
	Interpolation condition \eqref{generalintcondCross} holds by the definition of $a$ and $b$ and convexity of $f$. Friends condition \eqref{friendcond} follows from Proposition~\ref{freefriendslinprop} by noting that  $u = \bra{\pi,a \normss{\pi}} \in \lin\bra{C}$ and $\bra{\pi,0}^T u \neq 0$. The result then follows from Proposition~\ref{method1Cross}.
\qed
\end{proof}

\subsubsection{Non-separable positive homogeneous functions} \label{ENSF Section}

Proposition~\ref{method1Cross} can also be used to construct cuts for some non-separable functions, but as illustrated in the following example, we need slightly more complicated interpolations.  Consider $G:\Real\times \Real\to \Real$ given by $G(z,y)=\sqrt{z^2+y^2}$ and let $F$ be the elementary split associated with $\pi = e^1$, $\pi_0=-10$, and $\pi_1=1$. Constructing a parametric \emph{linear} interpolation as in \eqref{simple_splitcut_para} yields  
 \begin{equation}
	 J_L(z,y)=\frac{10 \sqrt{1 + y^2} + \sqrt{100 + y^2} + 
	 z \bra{\sqrt{1 + y^2} - \sqrt{100 + y^2}}}{11}.
 \end{equation}
 The associated cut is certainly valid, binding,  and sufficient for $\mybar{\epi\bra{G}}$ (we can always find friends by wiggling $z$ toward $\pi_0$ and $\pi_1$, and using $t$ to correct by following the slope of $J_L$ for fixed $y$). However, while $J$ is linear with respect to $z$, it is not convex with respect to $y$. We hence cannot use Proposition~\ref{method1Cross} for this interpolation. Fortunately, we can construct an alternative  interpolation given by 
 \begin{equation}\label{conicinterpolation}
	 J_C(z,y)={\sqrt{\bra{\frac{20 - 9 z}{11}}^2 +  y^2  }}
 \end{equation}
 that is convex on $(z,y)$. This function is not linear on $z$ for fixed  $y$, but we can still show it satisfies the interpolation condition \eqref{generalintcondCross} by noting that  $\bra{\frac{20 - 9 z}{11}}^2\leq z^2$ for any $z\notin (\pi_0,\pi_1)$ and that equality holds for $z\in \set{\pi_0,\pi_1}$. This is illustrated in Figure~\ref{firstpos2} for $y=-4$ where the graphs of $G$, $J_C$, and $J_L$ are given by the thick black curve, the thin blue curve, and the dash-dotted green line, respectively. The figure shows that  $J_C(z,y)\leq t$ is a nonlinear binding valid cut, but is strictly weaker than $J_L(z,y)\leq t$. While $J_C$ yields a weaker cut than $J_L$, $J_C$ is in fact the strongest \emph{convex} function that satisfies the interpolation condition \eqref{generalintcondCross} and we can show that $\mybar{\epi(G)}=\epi(G)\cap\epi(J_C)$. However, for the point $\bra{\overline{z},\overline{y},\overline{t}} \in \epi\bra{J_C}\cap \Int\bra{F}$ with $\overline{y}=-4$  depicted in Figure~\ref{firstpos2}, the friends construction  cannot be done by wiggling in a direction that leaves $\overline{y}$ fixed to $-4$. In other words, there are points in  $\hat{H}:=\set{\bra{z,y,t}\in \Real^3\,:\, y=-4}$ that do not have friends in $\hat{H}$. We can construct friends by wiggling in a direction that does change $\overline{y}$, but  since $\lin(\epi(J_C))=\emptyset$, such direction cannot be directly obtained from Proposition~\ref{freefriendslinprop}. Fortunately, the general idea of Proposition~\ref{freefriendslinprop} can be adapted to obtain a variant that directly reveals an appropriate direction.
\begin{figure}[htb]
\centering 
\includegraphics[scale=0.5]{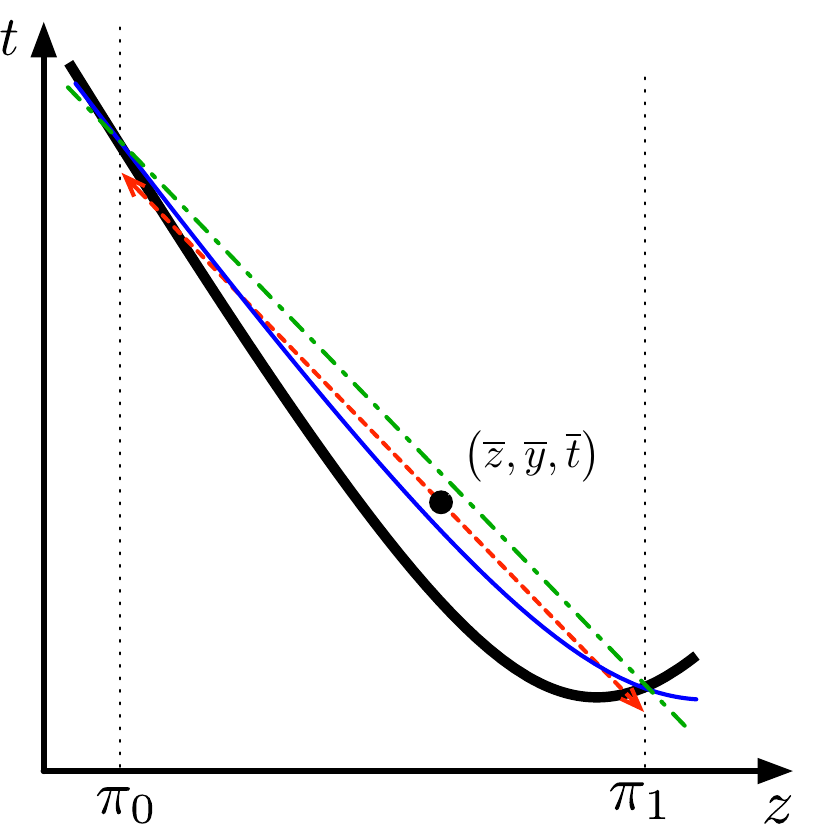}
\caption{Nonlinear interpolation for non-separable functions.}
\label{firstpos2}
\end{figure}

 The variant of Proposition~\ref{freefriendslinprop} that we need, exploits a different geometric characteristic of $\epi\bra{J_C}$ through the  generalization of  a technique used in \cite{Atamturk07,AN:conicmir}. The required geometric characteristic is given by the following definition.

\begin{definition} Let $C \subseteq\mathbb{R}^{n}$ be a closed convex set. We say $C$ is a \emph{translated cone} or \emph{conic set} if there exists $\xstar \in C$  such that $C-\xstar  $ is a convex cone. We refer to such $\xstar $ as \emph{an apex} of $C$, noting that it is not necessarily unique (e.g. a half space is a conic set whose apex is not unique).
\end{definition}

One can check that $\epi\bra{J_C}$ is a conic set with the unique apex $\bra{\zstar,\ystar ,\tstar } = \bra{20/9,0,0}$. Hence, because $(\bar{z},\mybar{y},\tbar)\in \epi\bra{J_C}$, we have that the ray
\beq
R:=\set{\bra{\zstar,\ystar ,\tstar } +\alpha \bra{\bra{\bar{z},\mybar{y},\tbar}-\bra{\zstar,\ystar ,\tstar }}\,:\, \alpha\geq0}\subseteq \epi\bra{J_C}. \label{ray}
\eeq
Furthermore, because $\zstar>\pi_1$ and $\bar{z}\in (\pi_0,\pi_1)$, there exists $\alpha_i>0$ such that $\zstar+\alpha_i (\bar{z}-{\zstar})=\pi_i$ for each $i\in \set{0,1}$. Therefore the friends of $(\bar{z},\mybar{y},\tbar)$ are given by $\bra{z^i,y^i,t^i}:=\bra{\zstar,\ystar ,\tstar }+\alpha_i \bra{(\bar{z},\mybar{y},\tbar)-\bra{\zstar,\ystar ,\tstar }}$ for $i\in \set{0,1}$. 

\begin{figure}[htb]
\centering 
\subfigure[Construction in the $\bra{z,y,t}$ space.]{\includegraphics[scale=0.5]{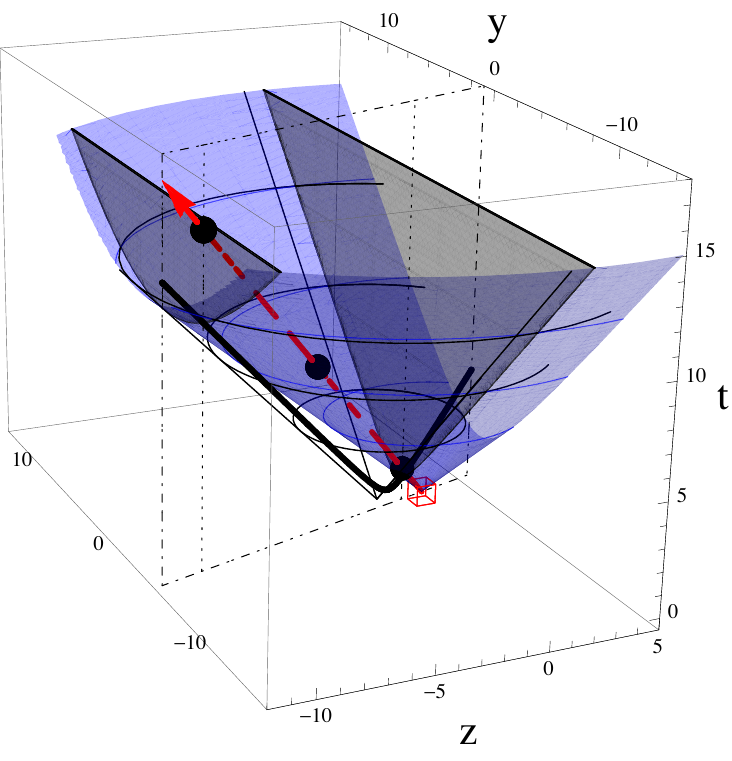}\label{sectionfig2}}
\subfigure[Construction in the hyperplane $\tilde{H}$.]{\includegraphics[scale=0.4]{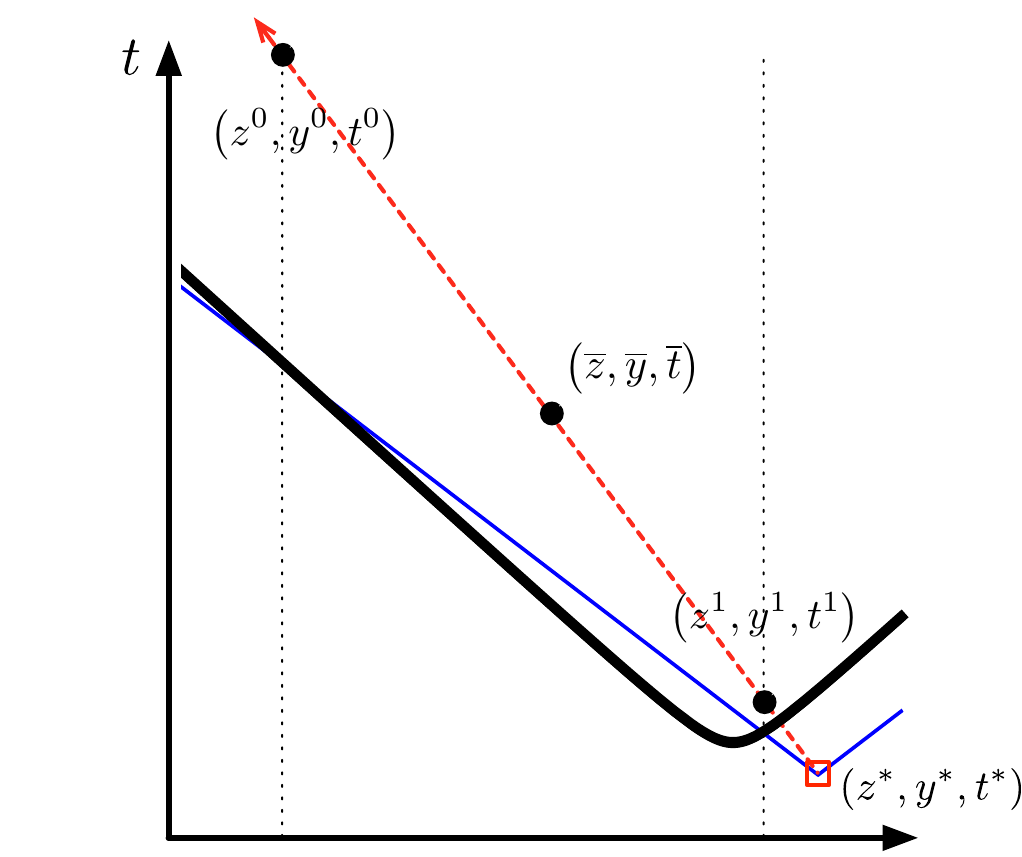}\label{sectionfig3}}
\caption{Friends construction for non-separable positive homogeneous functions.}\label{conefig}
\end{figure}
Figure~\ref{conefig} illustrates the ray-based friends construction for $(\bar{z},\mybar{y},\tbar)$ with $\mybar{y}=-4$.  Figure~\ref{sectionfig2} shows the construction in the $\bra{z,y,t}$ space, while Figure~\ref{sectionfig3} shows the section obtained by intersecting Figure~\ref{sectionfig2} with the hyperplane $\tilde{H}:=\aff\bra{R\cup\set{(0,0,1)}}$, for the ray $R$ given in \eqref{ray}.  The intersection of $\tilde{H}$ with the bounding box is depicted by the dash-dotted line in Figure~\ref{sectionfig2}. The graph of $G$ is given by a black wire-frame in Figure~\ref{sectionfig2}, while the intersection of this graph with $\tilde{H}$ is given by the thick black curve in both figures. Meanwhile, the graph of $J_C$ is depicted by the blue shaded region in Figure~\ref{sectionfig2} and by a thin blue curve in Figure~\ref{sectionfig3}. The figures also depict $\bra{z^i,y^i,t^i}$ for $i\in \set{0,1}$ and $(\bar{z},\mybar{y},\tbar)$ as black dots and $\bra{\zstar,\ystar ,\tstar }$ as a red box. In addition, the intersection of $z=\pi_i$ for $i\in\set{0,1}$ with the epigraphs of both $G$ and $J_C$ are depicted in Figure~\ref{sectionfig2} by the gray shaded regions. The intersection of $z=\pi_i$ for $i\in\set{0,1}$ with $\tilde{H}$ are depicted in both figures by dotted lines. Finally, ray R is depicted in both figures as a red dashed arrow.  Note that $\tilde{H}$ is tilted in the $\bra{z,y}$ space precisely to contain $\bra{\zstar,\ystar ,\tstar }$ and $\bra{\bar{z},\mybar{y},\tbar}$. Noting that  $y^* \neq \mybar{y}$ we have that, unlike $\hat{H}$, $\tilde{H}$ allows the variation of $y$. Furthermore, while $(\bar{z},\mybar{y},\tbar)\in\hat{H}\cap \tilde{H}$ might not have friends in $\hat{H}$, Figure~\ref{conefig} shows that it does have friends in $\tilde{H}$.

Similarly to Proposition~\ref{freefriendslinprop}, the above conic friends construction can be extended to general convex sets as follows.  
\begin{proposition}\label{freefriendsconeprop}
Let $F \subseteq \Realn$ be a split. If $C \subseteq \Realn$ is a conic set with  apex $\xstar \in \Real^n$ such that $\pi^T \xstar \notin \bra{\pi_0,\pi_1}$, then  condition  \eqref{friendcond} in Proposition~\ref{method1Cross} is satisfied.
\end{proposition}

\begin{proof}
Let $\xbar \in C$ such that $\pi^T \xbar \in \bra{\pi_0,\pi_1}$. Note that since $\xstar$ is the apex of $C$, all points on the ray $ R := \set{ x^{*} + \alpha \bra{\xbar - \xstar} : \alpha \in \mathbb{R}_+ }$
belong to $C$. Let the intersections of $R$ with the hyperplanes $\pi^T x = \pi_0$ and $\pi^T x = \pi_1$ be $x^0$ and $x^1$, respectively. Such points are obtained from $R$ by setting
\beqn
\alpha_i = \frac{\pi_i - \pi^T \xstar}{\pi^T \xbar - \pi^T \xstar},
\eeqn
for $i \in \set{0,1}$. We have $x^i \in C \cap    \bd\bra{F}$ for $i \in \set{0,1}$, since $\pi^T x^i = \pi_i$ and $R\subseteq C$.
Note that $\xbar$ is obtained from $R$ by setting $\alpha=1$. If $\alpha_0 < 1 < \alpha_1$ or  $\alpha_1 < 1 < \alpha_0$, then there exists $\beta \in (0, 1)$ such that $\xbar = \beta x^0 + \bra{1-\beta} x^1$. Seeing that $\pi^T \bar{x} \in \bra{\pi_0, \pi_1}$ and $\pi^T \xstar \notin \bra{\pi_0, \pi_1}$, one can check $\alpha_0 < 1 < \alpha_1$ or $\alpha_1 < 1 < \alpha_0$.
\qed
\end{proof}


Note that Propositions~\ref{freefriendslinprop} and \ref{freefriendsconeprop} ask for very different requirements on $C$. In Proposition~\ref{freefriendslinprop}, we only need to have a direction $u \in \lin\bra{C}$ such that $\pi^T u \neq 0$. In such case, $C$ always defines a non-pointed region (i.e., $C$ contains a line). On the other hand, as illustrated by \eqref{conicinterpolation}, the sets $C$ for which Proposition~\ref{freefriendsconeprop} is applicable are usually pointed (i.e. $C$ has at least one extreme point). However, pointedness is not a requirement in Proposition~\ref{freefriendsconeprop} (e.g. half-spaces are conic sets). The real price of Proposition~\ref{freefriendsconeprop} over Proposition~\ref{freefriendslinprop} is requiring $C$ to be conic, which is a much more global requirement than asking for the lineality space of $C$ to contain a non-orthogonal direction to $\pi$. However, both propositions are needed to construct split cuts for positive homogeneous functions.  To see this, consider the same function $G(z,y)=\sqrt{z^2+y^2}$ for which \eqref{conicinterpolation} yields a split cut, but instead consider the split $z\in  [-1,1]$. For this case, we can check that $\mybar{\epi(G)}=\epi(G)\cap\epi(J_D)$ for $J_D(z,y)=\sqrt{1+y^2}$, which does not have a conic epigraph. However, $(1,0,0)\in \lin(\epi(J_C))$ and hence Proposition~\ref{freefriendslinprop} is applicable. This dichotomy between a non-pointed and a conic (and potentially pointed) cut will be a common occurrence that we highlight further when characterizing intersection cuts for quadratic sets in Section~\ref{intquadsec}.

While Propositions~\ref{freefriendslinprop} and \ref{freefriendsconeprop} can be used to prove sufficiency of the split cuts for positive homogeneous functions, such cuts first have to be constructed with an appropriate interpolation technique. Fortunately, both interpolations of $G(z,y)=\sqrt{z^2+y^2}$ (conic and non-pointed)  can be generalized to functions based on $p$-norms  by using the following simple lemma whose proof is included in the appendix.

\begin{restatable}{lemma}{ghyperlemma} \label{ghyper interpolation lemma}
 Let $p \in {\mathbb{N}}$,  $\pi_0,\pi_1 \in \Real$ such that $\pi_0 < \pi_1$, $l \in \Real$, $a = \frac{\bra{\abs{l}^p + \abs{\pi_1}^p}^{1/p} - \bra{\abs{l}^p + \abs{\pi_0}^p}^{1/p}}{\pi_1 - \pi_0}$, and $b = \frac{\pi_1 \bra{\abs{l}^p + \abs{\pi_0}^p}^{1/p} - \pi_0 \bra{\abs{l}^p + \abs{\pi_1}^p}^{1/p}}{\pi_1 - \pi_0}$.
\begin{itemize}
\item If $s \in \set{\pi_0, \pi_1}$, then $\abs{as +b}^p = \abs{s}^p + \abs{l}^p$ and
	\item if $s \notin \bra{\pi_0, \pi_1}$, then $\abs{as +b}^p \le \abs{s}^p + \abs{l}^p$.
\end{itemize}
\end{restatable}



Using this lemma we can construct split cuts for epigraphs of a wide range of positive homogeneous convex functions and their sections (i.e. the epigraphs of such positive homogeneous functions after a variable is fixed to a constant).

\begin{restatable}{proposition}{ghyperprop}\label{ghyperprop}
	Let $F$ be a split, $\beta \in \mathbb{R}$, $l \in \Real$, $p\in \mathbb{N}$,  $g:\mathbb{R}^n\to \mathbb{R}$ be a positive homogeneous closed convex function, $a$ and $b$ as in Lemma~\ref{ghyper interpolation lemma},
and
\[H_{p,g} := \set{\bra{x, t} \in \Realnone : \left(g\left(P_\pi^{\perp} x\right)^p + \abs{\beta \pi^Tx}^p + \abs{\beta l}^p \right)^{1/p}\le t}.\] Then $\mybar{H_{p,g}} = H_{p,g} \cap C$, where 
\[ C = \set{ (x,t) \in \Realnone : 
	 \left(g\left(P_\pi^{\perp} x\right)^p + \left|\beta \bra{a\pi^T x+b}\right|^p \right)^{1/p}\le t}.\]
\end{restatable}
\begin{proof}
Interpolation condition \eqref{generalintcondCross} holds by the definition of $a$ and $b$ and Lemma~\ref{ghyper interpolation lemma}. If $\abs{\pi_0 } = \abs{\pi_1 }$, then $\bra{\pi,0} \in \lin\bra{C}$ and friends condition \eqref{friendcond} follows from Propositions~\ref{freefriendslinprop}. If $\abs{\pi_0} \ne \abs{\pi_1}$, then $C$ is a conic set with apex $\bra{x^*,t^*} = \bra{\frac{-b}{a \normss{\pi}} \pi,0}$. Furthermore, 
\beqn
\bra{\pi,0}^T \bra{x^*,t^*} = \pi^T x^* = \pi_1 + \bra{\abs{l}^p + \abs{\pi_1}^p}^{1/p} \rho = \pi_0 + \bra{\abs{l}^p + \abs{\pi_0}^p}^{1/p} \rho,
\eeqn
where $\rho = \frac{\pi_0 - \pi_1}{\bra{\abs{l}^p + \abs{\pi_1}^p}^{1/p} - \bra{\abs{l}^p + \abs{\pi_0}^p}^{1/p}}$. If $\abs{\pi_1} < \abs{\pi_0}$, then $\pi^T x^* \ge \pi_1$ and if $\abs{\pi_1} > \abs{\pi_0}$, then $\pi^T x^* \le \pi_0$. Therefore, friends condition \eqref{friendcond} follows from Proposition~\ref{freefriendsconeprop}. The result then follows from Proposition~\ref{method1Cross}.
\qed
\end{proof}

The following direct corollary of Proposition~\ref{ghyperprop} yields simplified formulas for split cuts when $l=0$ and  $H_{p,g}$ is the epigraph of a positive homogeneous convex function.  

\begin{restatable}{corollary}{pordersplit}\label{porder split}
	Let $F$ be a split, $\beta \in \mathbb{R}$, $p\in \mathbb{N}$,  $g:\mathbb{R}^n\to \mathbb{R}$ be a positive homogeneous closed convex function, $a = \frac{\pi_0 + \pi_1}{\pi_1 - \pi_0}$, $b = - \frac{2 \pi_1 \pi_0}{\pi_1 - \pi_0}$,
and
\[C_{p,g} := \set{\bra{x, t} \in \Realnone : \left(g\left(P_\pi^{\perp} x\right)^p + \abs{\beta \pi^Tx}^p \right)^{1/p}\le t}.\] If $0\notin (\pi_0,\pi_1)$, then $\mybar{C_{p,g}}=C_{p,g}$. Otherwise, 
$\mybar{C_{p,g}} = C_{p,g} \cap C$, where
\[ C = \set{ (x,t) \in \Realnone : 
	 \left(g\left(P_\pi^{\perp} x\right)^p + \left|\beta \bra{a \pi^T x+b}\right|^p \right)^{1/p}\le t}.\]
\end{restatable}
%

In particular, if $g$ is a $p$-norm and the splits are elementary, Corollary~\ref{porder split} further specializes as follows. 
\begin{corollary}\label{p-order cone split cut}
Let $F$ be an elementary split associated with $\pi = e^k$, $\| x \|_p = \bra{\sum \nolimits_{i=1}^{n} \abs{x_i }^p}^{1/p}$,  
\[K_p := \{ (x,t) \in \Realnone :  \|x\|_p\leq t \},\]
$a$ and $b$ as in Corollary~\ref{porder split},
and $\widehat{A}:=I-e^k {e^k}^T$.
If $0\notin (\pi_0,\pi_1)$, then $\mybar{K_p}=K_p $. Otherwise, $\mybar{K_p} = K_p \cap C$, where 	
 \[C = \set{\bra{x,t} \in \Realnone \,:\, \left\lVert\bra{\widehat{A}+ae^k {e^k}^T}x+b e^k\right\rVert_p\leq t	  }.\]
\end{corollary}
\begin{proof}
Direct from Corollary~\ref{porder split} by noting that
\beqn
K_p = \set{\bra{x,t} \in \Realnone \,:\, \bra{\left\lVert\widehat{A}x\right\rVert_p ^p + \abs{x_k}^p}^{1/p}\leq t	  },
\eeqn
and
\beqn
C = \set{\bra{x,t} \in \Realnone \,:\, \bra{\left\lVert\widehat{A}x\right\rVert_p^p + \abs{a x_k + b}^p}^{1/p}\leq t	  },
\eeqn
and seeing that $\widehat{A} = P_\pi^{\perp}$.
\end{proof}


 \subsection{Split cuts for level sets}

The interpolation technique can also be applied to some non-epigraphical sets. This is illustrated in the following proposition. 
\begin{proposition} \label{boundedSpli}
	Let $F$ be a split, $g:\mathbb{R}^n\to \mathbb{R}$ be a positive homogeneous convex function, $f:\mathbb{R}\to \mathbb{R}\cup\{+\infty\}$ be a closed convex  function such that $f(\pi_0), f(\pi_1) \le 0$, 
\[
L_{g,f} := \set{x \in \Real^n  : g\left(P_\pi^{\perp} x\right) + f(\pi^T x) \le 0},\] $a = \frac{f(\pi_1) - f(\pi_0)}{\pi_1 - \pi_0}$, and $b = \frac{\pi_1 f(\pi_0) - \pi_0 f(\pi_1)}{\pi_1- \pi_0}$.
 Then $\mybar{L_{g,f}} = L_{g,f} \cap C$, where
 \begin{equation}
	 C = \set{ x \in \Realn \,:\, g\left(P_\pi^{\perp} x\right)  +a \pi^T x+b\le 0}.
	 \label{homogeq}
 \end{equation}
\end{proposition}

\begin{proof}
Interpolation condition \eqref{generalintcondCross} holds by the definition of $a$ and $b$ and convexity of $f$. If  $f(\pi_0)=f(\pi_1)$, then $\pi \in \lin\bra{C}$ and friends condition \eqref{friendcond} follows from Proposition~\ref{freefriendslinprop}. If  $f(\pi_0) \neq f(\pi_1)$, then  $C$ is a conic set with  apex  $x^*=\frac{-b}{a \normss{\pi}}\pi$. Furthermore, 
\[ \pi^T x^* =\frac{\pi_0 f(\pi_1) -\pi_1 f(\pi_0)}{f(\pi_1)-f(\pi_0)}= \pi_1+\frac{(\pi_0-\pi_1)f(\pi_1)}{f(\pi_1)-f(\pi_0)}=\pi_0+\frac{(\pi_0-\pi_1)f(\pi_0)}{f(\pi_1)-f(\pi_0)}.\]
If $f(\pi_0)<f(\pi_1)$, then $\pi^T x^*\geq \pi_1$ and if $f(\pi_0)>f(\pi_1)$ then $\pi^T x^*\leq \pi_0$. Therefore, friends condition \eqref{friendcond} follows from Proposition~\ref{freefriendsconeprop}. The result then follows from Proposition~\ref{method1Cross}.
\qed
\end{proof}

As  a direct corollary of Proposition~\ref{boundedSpli}, we obtain formulas for elementary split cuts for balls of $p$-norms. 

\begin{corollary}\label{p-order ball split cut}
Let $F$ be an elementary split associated with $\pi = e^k$, $r \in \Real$ such that and $\abs{\pi_0},\abs{\pi_1}\leq r$, 
\[E_p := \{ x \in \mathbb{R}^n \ :  \|x\|_p\leq r \},\]
$f(u):=-\bra{r^p-\abs{u }^p}^{1/p}$, $a = \frac{f(\pi_1) - f(\pi_0)}{\pi_1 - \pi_0}$, $b = \frac{\pi_1 f(\pi_0) - \pi_0 f(\pi_1)}{\pi_1- \pi_0}$
,
and $\widehat{A}:=I-e^k {e^k}^T$.
Then $\mybar{E_p} = E_p \cap C$, where
 \[C = \set{ x  \in \Realn  :	\left\lVert\widehat{A}x\right\rVert_p+ax_k+b\leq 0	  }.\]
\end{corollary}
\begin{proof}
	Direct from Proposition~\ref{boundedSpli} by noting that \[E_p =\set{ x \in \mathbb{R}^n \ :  \left\lVert\widehat{A}x\right\rVert_p+f(x_k)\leq 0 }\]
and $\widehat{A} = P_\pi^{\perp}$.
\qed\end{proof}

\subsection{Non-trivial extensions}

In this section we consider two non-trivial extensions/applications of the interpolation technique. The first example considers t-inclusive split cuts for epigraphical sets and illustrates the case when the interpolation coefficients cannot be easily calculated. The second example shows how the technique can be used beyond split sets to  construct k-branch split cuts for epigraphical sets. We hope these examples serve as a guide for future applications or  extensions of the interpolation technique.

\subsubsection{t-inclusive split cuts for epigraphical sets}\label{tincsec}
Consider the base set  $Q_0 = \set{\bra{x,t} \in \Real^2 \,:\, x^2 \le t}$ and the t-inclusive split $x+t\in [0,1]$.  The first step to construct the associated split cut  $C \subseteq \Real^2$ such that $\mybar{Q_0} = Q_0 \cap C$ is to find the general form of such cut. The inclusion of $t$ in the split prevents us from directly using the interpolation arguments for regular splits to construct this general form. However, by extrapolating these arguments to the t-inclusive setting and analyzing the geometry of the problem (e.g. the intersection of $Q_0$ with $x+t\in \set{0,1}$ corresponds to two ellipses), we may guess that the appropriate interpolation form is 
\beq
C = \set{\bra{x,t} \in \Real^2 \,:\, \sqrt{\bra{ax+b}^2} \le cx+dt+e}, \label{intform}
\eeq
for some interpolation coefficients $a,b,c,d, e \in \Real$. Unlike  the regular split setting, it is not immediately  clear what these coefficients should be, but we may try to deduce them by forcing  interpolation conditions \eqref{generalintcondCross}.  Interpolation condition \eqref{generalintcondbCross} corresponds to 
\bear
\set{\bra{x,t} \in Q_0 \,:\, t = -x} &=& \set{\bra{x,t} \in C \,:\, t = -x} \label{tex1}\\
\set{\bra{x,t} \in Q_0 \,:\, t = 1-x} &=& \set{\bra{x,t} \in C \,:\, t = 1-x} \label{tex2},
\eear
which induces an infinite number of constraints on the coefficients.\footnote{For instance, \eqref{tex1} implies $\sqrt{\bra{ax+b}^2} \le \bra{c-d}x+e$ for all $\bra{x,-x} \in Q_0$.}
We could try to reduce such set of constraints to find the interpolation coefficients. In particular, the arguments for the regular splits effectively reduce such set of constraints to two equality constraints. For instance, in the interpolation given in \eqref{simple_splitcut}, the corresponding interpolation conditions analogous to \eqref{tex1} and \eqref{tex2} reduce to $G\bra{\pi_i} = J\bra{\pi_i}$ for $i \in \set{0,1}$. To obtain a similar reduction, we here take a possibly  naive approach that, nonetheless,  is successful for several classes of cuts  and is flexible enough to be extended to more complicated base and forbidden sets. The idea of this approach is to note that  \eqref{tex1} and \eqref{tex2} can be expressed as
\bear
\set{x \in \Real \,:\, x^2 \le -x} &=& \set{x \in \Real \,:\, \bra{ax+b}^2 \le \bra{\bra{c-d}x + e}^2, \bra{c-d}x + e \ge 0} \label{tex3}\\
\set{x \in \Real \,:\, x^2 \le 1-x} &=& \set{x \in \Real \,:\, \bra{ax+b}^2 \le \bra{\bra{c-d}x + d + e}^2, \bra{c-d}x + d + e \ge 0}. \label{tex4}
\eear
A sufficient condition for these constraints is for the quadratic polynomials in both sides of \eqref{tex3} and \eqref{tex4} to be identical, and for the following condition to hold:
\bear
\set{x \in \Real \,:\, x^2 \le -x} &\subseteq& \set{x \in \Real \,:\, \bra{c-d}x + e \ge 0} \label{nonneg1}\\
\set{x \in \Real \,:\, x^2 \le 1-x} &\subseteq& \set{x \in \Real \,:\, \bra{c-d}x + d + e \ge 0}.\label{nonneg2}
\eear
Forcing the polynomials to be identical is a simple matter of matching coefficients, which results in the following set of polynomial inequalities on $a,b,c,d,$ and $e$.
\bearn
a^2 - \bra{c-d}^2 &=& 1\\
ab - \bra{c-d}e &=& 1/2\\
ab - \bra{c-d}\bra{d+e} &=& 1/2\\
b^2 - e^2 &=& 0\\
b^2 - \bra{d+e}^2 &=& -1.
\eearn

The above linear system has four  solutions given by $\bra{1,\frac{1}{2},\frac{\sqrt{5}-1}{2},\frac{\sqrt{5}-1}{2},\frac{1}{2}}$, $\bra{1,\frac{1}{2},\frac{-\sqrt{5}+1}{2},\frac{-\sqrt{5}+1}{2},\frac{-1}{2}}$, $\bra{1,\frac{1}{2},\frac{\sqrt{5}+1}{2},\frac{\sqrt{5}+1}{2},\frac{-1}{2}}$, and $\bra{1,\frac{1}{2},\frac{-\sqrt{5}-1}{2},\frac{-\sqrt{5}-1}{2},\frac{1}{2}}$, of which only the first satisfies the additional  conditions \eqref{nonneg1} and \eqref{nonneg2}. Note that since $c=d$ in the first solution, checking \eqref{nonneg1} and \eqref{nonneg2} is equivalent to checking $e\ge0$ and $d+e\ge0$, which is trivial. Furthermore, this point also satisfies the interpolation condition \eqref{generalintcondcCross} which in this case, corresponds to 
\beq
\set{\bra{x,t} \in Q_0 \,:\, x+t \notin (0,1)} \subseteq \set{\bra{x,t} \in C \,:\, x+t \notin (0,1)}. \label{tex5}
\eeq
Finally, to show that this choice of interpolation coefficients yields the desired split cut, note that $C$ for such coefficients is a conic set  with apex  $\bra{x^*,t^*}=\bra{\frac{-1}{2},\frac{\sqrt{5}-3}{2\sqrt{5}-2}}$ and $x^*+t^*<0$. Then 
 friends condition \eqref{friendcond} follows from Proposition~\ref{freefriendsconeprop}. 

Note that identifying the coefficients of the quadratic polynomials  and having \eqref{nonneg1} and \eqref{nonneg2} are sufficient  for  interpolation condition \eqref{generalintcondbCross}, but they may not be necessary in general. Hence, there might be other interpolation coefficients for which $\mybar{Q_0} = Q_0 \cap C$. Moreover, it is not even clear that \eqref{intform} is the only 
possible interpolation form for the associated split cut. However, if the described procedure is successful, we need not worry about alternative characterizations, since they will all yield $\mybar{Q_0}$ when intersected with $Q_0$. There is of course no guarantee that the above procedure for finding a representation of $C$ will always succeed. However, as we illustrate in Section~\ref{intquadsec}, the procedure is successful in constructing rather complicated cuts for quadratic sets.

\subsubsection{k-branch split cuts for epigraphical sets}
We now illustrate how  Proposition~\ref{method1Cross} can be used for forbidden sets other than splits by constructing certain k-branch split cuts for separable functions. The following proposition is a direct, but rather technical, generalization of Proposition~\ref{GFormSplitCut}, which explains our reasoning to postpone its introduction to this stage of the paper. 
\begin{proposition} \label{GFormSplitCutCross}
	Let $g:\Real \to\Real$ and  $f_i:\Real \to\Real$ for each $i\in[k]$ be closed convex functions. Furthermore, let $F$ be a k-branch split such that $\pi_i \perp \pi_j$ for every $i\neq j$. Finally, let $P_{\Pi}^ {\perp} := I - \sum_{i=1}^k\frac{\pi_i \pi_i^T}{\|\pi_i \|_2^2}$, 
	 \beqn
		 B_{g,f} := \set{\bra{x,t} \in \Realnone \,:\, g\bra{P_{\Pi}^ {\perp} x}+ \sum_{i=1}^k f_i\bra{\pi_i^Tx} \le t},
	 \eeqn
	 $a_i := \frac{f_i\bra{\pi_1^i} - f_i\bra{\pi_0^i}}{\pi_1^i - \pi_0^i}$,  $b_i := \frac{\pi_1^i f_i\bra{\pi_0^i} - \pi^i_0 f_i\bra{\pi_1^i}}{\pi_1^i- \pi_0^i}$ for all $i\in [k]$,
	 and for every $\mathcal{I}\subseteq [k]$ let
	 \[h_{\mathcal{I}}(x):=g\bra{P_{\Pi}^ {\perp} x}+  {\sum_{i\in [k]\setminus \mathcal{I}} f_i\bra{\pi_i^Tx}+  \sum_{i\in \mathcal{I}} a_i\pi_i^T x+ b_i}.\]
	 Then $\mybar{B_{g,f}} = B_{g,f} \cap C$, where
	 \[C = \set{\bra{x,t} \in \Realnone \,:\, \max_{\mathcal{I}\subseteq [k]} h_{\mathcal{I}}(x) \le t}.\]
\end{proposition}
\begin{proof}
	 Interpolation condition \eqref{generalintcondCross} holds by the definition of $a_i$ and $b_i$ and convexity of $f_i$. Now let $\bra{\overline{x},\tbar}\in C \cap \Int\bra{F}$. To construct the friends of  $\bra{\overline{x},\tbar}$ we proceed as follows. 

	Let $\mathcal{I}\subseteq [k]$ be such that for all $i\in \mathcal{I}$ we have $\pi_i^T \xbar\in \bra{\pi_0^i,\pi_1^i}$,  and for all $i\in [k]\setminus \mathcal{I}$ we have $\pi_i^T \xbar \notin \bra{\pi_0^i,\pi_1^i}$.  For each $s\in \set{0,1}^{\mathcal{I}}$, let
\beq
x^s = P_{\Pi}^{\perp} \bar{x} + \sum_{i\in [k]\setminus \mathcal{I}} \frac{\pi_i^T \bar{x}}{\norm{\pi_i}_2^2} \pi_i+ \sum_{i\in \mathcal{I}} \frac{s_i \pi^i_0 +(1-s_i)\pi_1^i}{\norm{\pi_i}_2^2} \pi_i, \quad t^s = \tbar + \sum_{i\in \mathcal{I}} a_i \bra{s_i \pi^i_0 +(1-s_i)\pi_1^i- \pi^T_i \bar{x}},
\eeq
and
\beq
\lambda_s= \prod_{i\in \mathcal{I}}\bra{ s_i \frac{\pi_1^i-\pi_i^T \xbar}{\pi_1^i-\pi_0^i}+(1-s_i)\frac{\pi_i^T \xbar-\pi_0^i}{\pi_1^i-\pi_0^i}}.
\eeq
One can check that $\bra{\overline{x},\tbar}=\sum_{s\in \set{0,1}^{\mathcal{I}}} \lambda_s \bra{x^s, t^s}$, $\sum_{s\in \set{0,1}^{\mathcal{I}}} \lambda_s=1$, and $\lambda_s\geq 0 $ for all $s\in \set{0,1}^{\mathcal{I}}$. Furthermore, by construction and the assumption on $\mathcal{I}$, we have that $x^s\in \bd\bra{F}$ and $\bra{x^s, t^s}\in \epi\bra{h_{\mathcal{I}}}$ for all $s\in \set{0,1}^{\mathcal{I}}$. The result then follows from Proposition~\ref{method1Cross} by noting that  for all $s\in \set{0,1}^{\mathcal{I}}$, we have $\max_{\mathcal{J}\subseteq [k]} h_{\mathcal{J}}(x^s)=h_{\mathcal{I}}\bra{x^s}$.
\qed\end{proof}

\section{Intersection cuts for conic quadratic sets}\label{intquadsec}
In this section we consider intersection cuts for conic quadratic sets of the form $\mathcal{C} \eqq \set{x \in \Realn \,:\, Ax-d \in L^m}$ where $A \in \Real^{m \times n}$, $d \in \Real^m$, and  $L^m$ is the $m$-dimensional Lorentz cone. Note that $\mathcal{C}$ can be written as
\beq \label{cset}
\mathcal{C} = \set{x \in \Realn \,:\, \norms{A_0x - d_0} \le a_m^T x -d _m},
\eeq
where $\bra{A_0,d_0}$ is obtained from $\bra{A,d}$ by deleting the $m$-th row, and $\bra{a_m,d_m}$ is the $m$-th row of $\bra{A,d}$. Using \eqref{cset}, one can rewrite $\mathcal{C}$ as
\beqn
\mathcal{Q} \eqq \set{x \in \Realn \,:\, x^T Q x - 2h^T x + \rho \le 0, \quad a_m^T x - d_m\geq 0},
\eeqn
where $Q = A_0^T A_0 - a_m a_m^T$, $h = A_0^T d_0 - a_m d_m$, and $\rho = d_0^T d_0 - d_m^T d_m$. Also note that $Q \in \Real^{n \times n}$  is symmetric with at most one negative eigenvalue. Using  known classifications of sets described by a quadratic inequality with at most one negative eigenvalue (e.g. see Table~2.1 and the reasoning after the proof of Lemma~2.1 in \cite{belotti2013families}), we have that all conic quadratic sets of the form $\mathcal{C}$ correspond to the following list:
 \begin{enumerate}
	 \item A full dimensional paraboloid, 
	 \item a full dimensional ellipsoid (or a single point), 
	 \item a full dimensional second-order cone,
	 \item one side of a full dimensional hyperboloid of two sheets, 
	 \item a cylinder generated by a lower-dimensional version of one of the previous sets, or
		\item an invertible affine transformation of one of the previous sets.
 \end{enumerate}
 
We first consider split cuts for conic quadratic sets with simple structures  that can be obtained as direct corollaries of Propositions~\ref{GFormSplitCut}, \ref{ghyperprop}, and \ref{boundedSpli}. We then consider t-inclusive and k-branch split cuts for conic quadratic sets that require ad-hoc proofs based on Proposition~\ref{method1Cross}.
As expected, we see that split cut formulas are significantly simpler than those for t-inclusive and k-branch split cuts. However, in either case, it is crucial to exploit the symmetry of the Euclidean norm through the following standard lemma. 
\begin{lemma} \label{norm decomposition}
For $v \in \Real^n$,
$\|x\|_2^2 = \norm{P_v x }_2^2+\norm{P_{v}^{\perp}x}_2^2$.
\end{lemma}

To give formulas for split cuts for all the sets 1--6, it suffices to give formulas for the cases 1--4. With these, we can construct split cut formulas for cylinders using the following lemma, which we prove in the appendix.
\begin{restatable}{lemma}{cylinderlemma}

Let $B\subseteq \Real^n$ be a closed convex set of the form $B_0+L$ where $L$ is a linear subspace, and let $F\subseteq \Real^n$ be a split. If $\pi\in L^\perp$
 and $\conv\bra{B_0 \setminus \Int\bra{F}} = B_0 \cap C$, then $\conv\bra{B \setminus \Int\bra{F}} = \bra{B_0 \cap C} + L$. If  $\pi\notin L^\perp$, then  $\conv\bra{B \setminus \Int\bra{F}} = B$.
 \end{restatable}
 
Finaly, we can  construct split cut formulas for affine transformations by using the following straightforward lemma.
\begin{lemma}\label{basictransformlemma}
Let $B\subseteq \Real^n$ be a closed convex set, $F\subseteq \Real^n$ be a split, and $M: \Realn \rightarrow \Realn$ be an invertible affine mapping. If $\conv\bra{B \setminus \Int\bra{F}} = B \cap C$ for a closed convex set $C \subseteq \Realn$, then
\beqn
\conv\bra{M\bra{B} \setminus \Int\bra{M\bra{F}}} = M\bra{B} \cap M\bra{C}.
\eeqn
\end{lemma}

We note that classification 1--6 is not strictly necessary for constructing split cuts for quadratic sets. In particular, an algorithm introduced in \cite{yildiran2009convex} can be used to obtain an SDP representation of split cuts for any quadratic set (convex or not) without a priori classifying its specific geometry as in 1--6. However, the procedure in \cite{yildiran2009convex} requires the execution of a numerical algorithm to construct split cuts and does not provide closed form expressions of the cuts. Furthermore, such an algorithm requires elaborate algebraic tools specific to   quadratic sets that go far beyond a basic property such as that described by Lemma~\ref{norm decomposition}.  Hence, the objective of the following subsection is not to present the shortest possible constructions of all quadratic split cuts, but to (i) present simple proofs tailored to the specific geometries in classification 1--6 and (ii) present a case study on the power and limitations of the general interpolation approach to split cuts.

\subsection{Split cuts for quadratic sets}\label{simplequasplitcuts}

 Split cuts can be obtained for ellipsoids when interpreted as lower level sets of quadratic or conic functions (i.e., based on the Euclidean norm). Similarly, split cuts can also be characterized for paraboloids and cones that, when interpreted as epigraphs of quadratic or conic functions, are such that $t$ is unaffected by the split disjunctions.
We note that the ellipsoid case has already been proven on \cite{Goez2012,Dadush2011121}, and that the conic case generalizes Proposition~2 in \cite{AN:conicmir} which considers elementary disjunctions for the standard three dimensional second-order cone. 

\begin{corollary}[Split cuts for paraboloids]\label{quadratic coro}
Let $F$ be a split, $A \in \mathbb{R}^{n \times n}$ be an invertible matrix, $ c \in \mathbb{R}^n$, 
\[Q := \set{\bra{x,t} \in \Realnone : \|A\bra{x-c}\|_2^2 \leq t},\]  $
a = \frac{\pi_0 + \pi_1 - 2  \pi^T c }{\|A^{-T}\pi\|_2^2} $,
$
b = - \frac{\bra{\pi_1 - \pi^T c} \bra{\pi_0 - \pi^T c}}{\|A^{-T}\pi\|_2^2},
$
and $\widehat{A}= P_{A^{-T} \pi}^ {\perp} A$. Then $\mybar{Q} = Q \cap C$, where
\[C = \set{\bra{x,t} \in \Realnone:	  \left\lVert \widehat{A} \bra{x-c}\right\rVert_2^2 + a  \pi^T\bra{ x-c } + b \le t  }.\]
\end{corollary}
\begin{proof}
Note that for the affine mappings $M, M^{-1}$ given by $M(x) = A\bra{x-c}$ and $M^{-1}(x) = A^{-1}x+c$, we have $Q = M^{-1}\bra{Q_0}$ and $Q_0 = M\bra{Q}$, where $Q_0 = \set{(x,t)\in \Realnone\,:\, \normss{x} \le t}$. Using Lemma \ref{basictransformlemma}, we prove the corollary by finding a closed form expression for $\mybar{Q_0}$ where the forbidden set is the split $M\bra{F}$ associated with $\tilde{\pi} = A^{-T}\pi$, $\tilde{\pi}_0 = \pi_0 - \pi^T c$, and $\tilde{\pi}_1 = \pi_1 - \pi^T c$. By Lemma \ref{norm decomposition}, we have
\[Q_0 = \set{\bra{x,t} \in \Realnone : \| P_{\tilde{\pi}}^ {\perp} x \|_2^2 + \frac{(\tilde{\pi}^T x)^2}{\norm{\tilde{\pi}}_2^2} \le t}.\] The result then follows from Proposition~\ref{GFormSplitCut}.
\qed\end{proof}

\begin{corollary}[Split cuts for cones]\label{nlscut prop}
Let $F$ be a split, $A \in \mathbb{R}^{n \times n}$ be an invertible matrix, $ c \in \mathbb{R}^n$, 
\[K := \set{\bra{x,t} \in \Realnone : \|A\bra{x-c}\|_2 \leq t},\] $a = \frac{\pi_1 + \pi_0 -2  \pi^T c }{\pi_1 - \pi_0}
$,
$
b = \frac{-2 \bra{\pi_1 - \pi^T c} \bra{\pi_0 - \pi^T c}}{\pi_1 - \pi_0}
$, $\widehat{A}=\bra{P_{A^{-T} \pi}^ {\perp}+a P_{A^{-T} \pi}}  A$, $\widehat{c}=\bra{b/\norm{A^{-T} \pi}_2^2} A^{-T} \pi$. If $ \pi^T c\notin (\pi_0,\pi_1)$, then $\mybar{K} = K$. Otherwise, $\mybar{K} = K \cap C$, where
 \[C =\set{(x,t)  \in \Realnone:		
 \left\lVert \widehat{A} \bra{x-c}+\widehat{c}\, \right\rVert_2 \le t}.\]
\end{corollary}
\begin{proof}
Note that for the affine mappings $M, M^{-1}$ given by $M(x) = A\bra{x-c}$ and $M^{-1}(x) = A^{-1}x+c$, we have $K = M^{-1}\bra{K_0}$ and $K_0 = M\bra{K}$, where $K_0 = \set{(x,t)\in \Realnone\,:\, \norms{x} \le t}$. Using Lemma \ref{basictransformlemma}, we prove the corollary by finding a closed form expression for $\mybar{K_0}$ where the forbidden set is the split $M\bra{F}$ associated with $\tilde{\pi} = A^{-T}\pi$, $\tilde{\pi}_0 = \pi_0 - \pi^T c$, and $\tilde{\pi}_1 = \pi_1 - \pi^T c$. By Lemma \ref{norm decomposition}, we have 
\[K_0 = \set{\bra{x,t} \in \Realnone : \bra{\| P_{\tilde{\pi}}^ {\perp} x \|_2^2 + \frac{(\tilde{\pi}^T x)^2}{\norm{\tilde{\pi}}_2^2}}^{1/2} \le t}.\] The result then follows from Corollary~\ref{porder split}.
\qed
\end{proof}

A particularly interesting application of Corollaries~\ref{quadratic coro} and \ref{nlscut prop} is the Closest Vector Problem \cite{MGbook}, which can be alternatively written as $\min\set{\norm{A\bra{x-c}}_2^2\,:\, x\in \mathbb{Z}^n}$ or $\min\set{\norm{A\bra{x-c}}_2\,:\, x\in \mathbb{Z}^n}$. In turn, these problems can be reformulated as
\[ \min\set{t\,:\, \bra{x,t}\in Q,\,x\in \mathbb{Z}^n} \mbox{\quad and \quad } \min\set{t\,:\, \bra{x,t}\in K,\, x\in \mathbb{Z}^n},\]
respectively. We can then use Corollaries~\ref{quadratic coro} and \ref{nlscut prop} with lattice free splits to construct split cuts that could improve the solution speed of these problems. We are currently studying the effectiveness of such cuts.

We can also obtain as a corollary the following result from \cite{Goez2012,Dadush2011121}.

\begin{corollary}[Split cuts for ellipsoids] \label{ellip coro}
Let $F$ be a split, $A \in \mathbb{R}^{n \times n}$ be an invertible matrix, $ c \in \mathbb{R}^n$, $r \in \Realp$,
\[E := \set{x\in \mathbb{R}^n : \|A\bra{x-c}\|_2 \leq r},\]  $f(u):=-\sqrt{r^2-\frac{u^2}{\|A^{-T} \pi \|_2^2  }}$, $a = \frac{f(\pi_0-\pi^T c)-f(\pi_1-\pi^T c)}{\pi_1 - \pi_0}$, and \[b = \frac{\bra{\pi_1-\pi^T c} f(\pi_0-\pi^T c) - \bra{\pi_0-\pi^T c} f(\pi_1-\pi^T c)}{\pi_1- \pi_0}.\]
If $\pi^T c-r \norm{A^{-T}\pi}_2\leq {\pi_0 < \pi_1}\leq \pi^T c+r \norm{A^{-T}\pi}_2$, then $\mybar{E} = E \cap C$, where
\bear\label{splitproper}
C = \set{ x \in \Realn : 	\| P_{A^{-T} \pi}^ {\perp} A \bra{x-c} \|_2 \le  a \pi^T (x-c) - b },\,\quad
\eear
{if $\pi_0 < \pi^T c-r \norm{A^{-T}\pi}_2 < \pi_1\leq \pi^T c+r \norm{A^{-T}\pi}_2$, then}
\bear\label{plitcg1}
\mybar{E} = \set{  x \in E :  \pi^T x\geq \pi_1 },\eear
{if $\pi^T c-r \norm{A^{-T}\pi}_2\leq \pi_0 < \pi^T c+r \norm{A^{-T}\pi}_2 < \pi_1$, then}
\bear\label{plitcg2}
\mybar{E} = \set{  x \in E : \pi^T x\leq \pi_0 },\eear
{if $\pi^T c-r \norm{A^{-T}\pi}_2 \ge \pi_1$ or $\pi_0 \ge \pi^T c+r \norm{A^{-T}\pi}_2$, then $\mybar{E} = E$}, and otherwise, $\mybar{E} = \emptyset$.
\end{corollary}

\begin{proof}
Note that for the affine mappings $M, M^{-1}$ given by $M(x) = A\bra{x-c}$ and $M^{-1}(x) = A^{-1}x+c$, we have $E = M^{-1}\bra{E_0}$ and $E_0 = M\bra{E}$, where $E_0 = \set{(x,t)\in \Realnone\,:\, \norms{x} \le r}$. Using Lemma \ref{basictransformlemma}, we prove the corollary by finding a closed form expression for $\mybar{E_0}$ where the forbidden set is the split $M\bra{F}$ associated with $\tilde{\pi} = A^{-T}\pi$, $\tilde{\pi}_0 = \pi_0 - \pi^T c$, and $\tilde{\pi}_1 = \pi_1 - \pi^T c$. By Lemma \ref{norm decomposition}, we have 
\[E_0 = \set{x \in \mathbb{R}^n : \| P_{\tilde{\pi}}^ {\perp} x \|_2 - \sqrt{r^2 - \frac{(\tilde{\pi}^T x)^2}{\norm{\tilde{\pi}}_2^2}} \le 0}.\]
The result then follows from Proposition~\ref{boundedSpli}.

The other cases can be shown by studying when the ellipsoid is partially or completely contained in one side of the disjunction, or when it is completely contained strictly between the disjunction. 
\qed\end{proof}
We note that Corollary \ref{ellip coro} shows there are two types of  split cuts for $E$. In \eqref{splitproper}, we obtain a nonlinear split cut that we would expect from Proposition~\ref{boundedSpli}, while in \eqref{plitcg1}--\eqref{plitcg2} we obtain simple linear split cuts. These linear inequalities are  actually Chv\'atal-Gomory (CG) cuts for $E$ \cite{Chvatal73,DBLP:journals/mor/DadushDV11,DBLP:conf/ipco/DadushDV11,DBLP:conf/ipco/DeyV10,Gomory58}, but  they are still sufficient to describe $\mybar{E}$ together with the original constraint.  We hence follow the same MILP convention used in \cite{Dadush2011121} and still consider them  split cuts. Note that we can also consider ``CG split cuts'' in Proposition~\ref{boundedSpli} if we include additional structure on the functions such as $g$ being non-negative. Similarly, we can also do the case analysis for CG cuts in Corollary~\ref{p-order ball split cut}.

\begin{proposition}[Split cuts for hyperboloids]\label{simplehyper}
Let $F$ be a split, $l \in \Real\setminus \set{0}$,
\[H := \set{\bra{x,t} \in \Realnone : \sqrt{\normss{x} + l^2} \leq t},\]
$a = \frac{\sqrt{l^2 \normss{\pi} + \pi_1^2} - \sqrt{l^2 \normss{\pi} + \pi_0^2}}{\pi_1 - \pi_0}$, and
$b = \frac{\pi_1 \sqrt{l^2 \normss{\pi} + \pi_0^2} - \pi_0 \sqrt{l^2 \normss{\pi} + \pi_1^2}}{\pi_1 - \pi_0}$. Then $\mybar{H} = H \cap C$, where
\[C = \set{\bra{x,t} \in \Realnone : \norms{P_{\pi}^{\perp} x + \frac{a\pi^T x + b}{\normss{\pi}} \pi} \le t}.\]
\end{proposition}

\begin{proof}
Direct from Proposition~\ref{ghyperprop} by noting that
\beqn
H := \set{\bra{x,t} \in \Realnone : \sqrt{\normss{P_{\pi}^ {\perp}} + \frac{\bra{\pi^T x}^2}{\normss{\pi}} + l^2} \leq t}.
\eeqn
\qed
\end{proof}
%
%

\subsection{t-inclusive split cuts for quadratic sets}\label{complicated}

The split cut formulas in this section are significantly more complicated. For this reason, we only present them for standard sets (i.e., with $A=I$ and $c=0$). Formulas for the general case may be obtained by combining the formulas for the standard case with Lemma~\ref{basictransformlemma}.

\begin{restatable}{proposition}{tparaboloidprop}\textnormal{\textbf{(t-inclusive split cuts for paraboloids)}}\label{newpnlscut prop}
Let $F$ be a t-inclusive split and \[Q_0 := \{\bra{x,t} \in \Realnone : \|x\|_2^2 \leq t\}.\] If $\phat > 0$ and $\pi_1 \le \frac{-\normss{\pi}}{4\phat}$, or if $\phat < 0$ and $\frac{-\normss{\pi}}{4\phat} \le \pi_0$, then \[\mybar{Q_0} = Q_0,\]
if $\phat > 0$ and $\pi_0 < \frac{-\normss{\pi}}{4\phat} < \pi_1$, then \[\mybar{Q_0} = \set{(x, t) \in Q_0 : \pi^T x + \phat t \ge \pi_1},\]
if $\phat < 0$ and $\pi_0 < \frac{-\normss{\pi}}{4\phat} < \pi_1$, then \[\mybar{Q_0} = \set{(x, t) \in Q_0: \pi^T x + \phat t \le \pi_0},\] and if $\phat > 0$ and $\frac{-\normss{\pi}}{4\phat} \le \pi_0$, or if $\phat < 0$ and $\pi_1 \le \frac{-\normss{\pi}}{4\phat}$, then $\mybar{Q_0} = Q_0 \cap C$, where
\begin{equation*}\label{generalpa}
C = \set{ (x,t) \in \Realnone : \norm{P_\pi^\perp x + \frac{\pi^T x + b }{\norm{\pi}_2^2} \pi}_2 \le c \pi^T x + d t + e},
\end{equation*}
for
\bearn
&&b = \frac{\norm{\pi}_2^2}{2 \hat{\pi}} \\
&&c = \frac{ f }{\sqrt{2} \bra{\pi_1 - \pi_0}\hat{\pi}} \\
&& d = c \phat \\
&&e=  \frac{ \normss{\pi} + \sqrt{\normss{\pi} + 4 \pi_0 \hat{\pi}} \sqrt{\normss{\pi} + 4 \pi_1 \hat{\pi}} }{4 \sqrt{2} \bra{\pi_1 - \pi_0} \hat{\pi}^2} f\\
&&f = \sqrt{\norm{\pi}_2^2 + 2 \bra{\pi_0 + \pi_1} \hat{\pi} - \sqrt{\norm{\pi}_2^2 + 4 \pi_0 \hat{\pi}} \sqrt{\norm{\pi}_2^2 + 4 \pi_1 \hat{\pi}}},
\eearn
where we use the convention $0/0 \eqq 0$ for the case $\norms{\pi}=0$.
 \end{restatable}
\begin{proof}
See appendix.
\qed
\end{proof}

 
\begin{restatable}{proposition}{tconeprop}\textnormal{\textbf{(t-inclusive split cuts for cones)}}\label{newnlscut prop}
Let $F$ be a t-inclusive split and
\[K_0 := \{\bra{x,t} \in \Realnone : \|x\|_2 \leq t\}.\]  If $0 \notin (\pi_0, \pi_1)$, then $\mybar{K_0} = K_0$. Otherwise, if $0 \in (\pi_0, \pi_1)$ and $\hat{\pi} \le - \norm{\pi}_2$, then \[\mybar{K_0} = \set{ (x, t) \in K_0 : \pi^T x + \hat{\pi} t \le \pi_0},\] if $0 \in (\pi_0, \pi_1)$ and $\hat{\pi} \ge \norm{\pi}_2$, then \[\mybar{K_0} = \set{ (x, t) \in K_0 : \pi^T x + \hat{\pi} t \ge \pi_1},\] and if $0 \in (\pi_0, \pi_1)$ and $\hat{\pi} \in \bra{-\norm{\pi}_2, \norm{\pi}_2}$, then $\mybar{K_0} = K_0 \cap C$, where
\[C = \set{ (x, t) \in \Realnone : \norm{P_\pi^\perp x + \frac{a \pi^T x + b }{\norm{\pi}_2^2} \pi}_2 \le c \pi^T x + d t + e},\] where
\bearn
&&a = \frac{ \bra{\pi_0 + \pi_1} \bra{\norm{\pi}_2^2 - \hat{\pi}^2} } { f } \\
&&b = - \frac{ 2 \pi_0 \pi_1 \norm{\pi}_2^2} { f } \\
&&c = - \frac{ 4 \pi_0 \pi_1 \hat{\pi} }{ \bra{\pi_1 - \pi_0} f} \\
&&d = \frac{ f }{ \bra{\pi_1 - \pi_0} \bra{\normss{\pi} - \phat^2} }\\
&&e= \frac{ 2 \pi_0 \pi_1 \bra{\pi_0 + \pi_1} \hat{\pi} }{ \bra{\pi_1 - \pi_0} f }\\
&& f = \sqrt{ \bra{\norm{\pi}_2^2 - \hat{\pi}^2} \bra{ \norm{\pi}_2^2 \bra{\pi_1 - \pi_0}^2 - \hat{\pi}^2 \bra{\pi_0 + \pi_1}^2} }.
\eearn
 \end{restatable}
\begin{proof}
See appendix.
\qed
\end{proof}

With regards to the general interpolation forms of the obtained split cuts in Sections 4.1 and 4.2, we note that these fall into two categories. The first category corresponds to the case in which the intersection of the boundary of the split and the base set is bounded such as when the base set is an ellipsoid. In such case, the obtained split cuts are always an ellipsoidal cylinder or a conic set. The second category corresponds to the case in which the intersection of the boundary of the split and the base set is unbounded. In such case, the obtained split cut is of the same form as the base set. For instance,  split cuts for conic sets or sections of conic sets are conic. An nice illustration of this dichotomy is the case of paraboloids, where t-inclusive splits have bounded intersections and yield conic cuts, while splits that are not t-inclusive have unbounded intersections and yield parabolic cuts.

Finally, we note that the only formulas that we did not explicitly characterize here are t-inclusive split cuts for affine transformations of paraboloids and cones, split cuts for affine transformation of hyperboloids, and t-inclusive split cuts for hyperboloids and their affine transformations. All such formulas can be obtained using Lemma~\ref{basictransformlemma}, except t-inclusive split cuts for hyperboloids. We can still obtain  formulas  for t-inclusive split cuts for hyperboloids using the interpolation technique; however, the resulting formulas are significantly more involved and no longer fit the  ``simple'' formulas theme of the paper. However, the analysis so far is still a significant generalization of what is known for split cuts for conic quadratic sets. In fact, the most general alternative that we are aware of is the concurrently developed technique in \cite{Kent}, which consider conic sets of the form $\set{x \in \Realn \,:\, Ax-d \in L^m}$ for a full rank matrix $A$,  which we do not require. When $A$ does not have full row rank, it is possible to consider a full row rank submatrix of $A$ and use this relaxation to generate the cuts from \cite{Kent}. However, as noted in Example~1 of \cite{Kent}, this approach fails to give split cuts for hyperboloids which we can obtain from Proposition~\ref{simplehyper} and Lemma~\ref{basictransformlemma}. Nevertheless, one advantage of the approach in \cite{Kent} is the use of a more systematic procedure to obtain the interpolation coefficients, which can be particularly useful when constructing t-inclusive split cuts. For instance, in Proposition~\ref{newnlscut prop} we obtain the interpolation coefficients through the heuristic procedure described in Section~\ref{tincsec}, which required guessing the interpolation form of the split cut and was not guaranteed to be successful even if this guess was accurate. In contrast, the approach in \cite{Kent} only assumes that the split cut is a polynomial inequality and calculates the coefficients of the associated polynomial through a systematic use of techniques from algebraic geometry. The conversion of this polynomial inequality to a conic quadratic inequality is an ad-hoc procedure that might be limited to quadratic cones. However, the construction of the initial polynomial inequality seems to have a higher chance of being extended to higher order cones or more general semi-algebraic sets than the approach in  Section~\ref{tincsec}. In contrast, when we consider split disjunctions that are not t-inclusive, the approach from Section~\ref{ENSF Section} has an advantage as it is not restricted to semi-algebraic sets.

\subsection{k-branch split cuts for quadratic sets}

Similarly to Corollary~\ref{quadratic coro}, we can use the following direct generalization of Lemma~\ref{norm decomposition} to get formulas for several families of k-branch split cuts for convex quadratic sets. 

\begin{lemma} \label{norm decompositioncross}
	Let $\set{\pi_i}_{i=1}^k \subseteq \Real^n\setminus\set{0_n}$ be such that $\pi_i \perp \pi_j$ for every $i\neq j$ and $P_{\Pi}^ {\perp} := I - \sum_{i=1}^k\frac{\pi_i \pi_i^T}{\|\pi_i \|_2^2}$. Then for any $v \in \Real^n$
	we have \[\|x\|_2^2 =\norm{P_{\Pi}^{\perp}x}_2^2 +\sum_{i=1}^k\frac{\bra{\pi^T_i x}^2}{\norm{\pi_i}^2_2}.\]
\end{lemma} 
 The following corollary generalizes the result of Corollary~\ref{quadratic coro} to the case of k-branch split cuts for paraboloids.

\begin{corollary}[k-branch split cuts for paraboloids]\label{quadratic corocross}
	Let $A \in \mathbb{R}^{n \times n}$ be an invertible matrix, $ c \in \mathbb{R}^n$, and
\beqn
Q := \set{\bra{x,t} \in \Realnone : \|A\bra{x-c}\|_2^2 \leq t}.
\eeqn
Also let $F$ be a k-branch split such that $A^{-T}\pi_i \perp A^{-T} \pi_j$ for every $i\neq j$,
$a_i = \frac{\pi_0^i + \pi_1^i - 2  \pi_i^T c }{\|A^{-T}\pi_i\|_2^2} $ and
$b_i = - \frac{\bra{\pi_1^i - \pi^T_i c} \bra{\pi_0^i - \pi^T_i c}}{\|A^{-T}\pi_i\|_2^2}$
for all $i\in [k]$,  and for every $\mathcal{I}\subseteq [k]$ let
\[h_{\mathcal{I}}(x):=\left\lVert \bra{A - \sum_{i \in \mathcal{I}}\frac{A^{-T}\pi_i \pi_i^T}{\|A^{-T}\pi_i \|_2^2}}\bra{x-c}\right\rVert_2^2+ \sum_{i\in \mathcal{I}} a_i\pi_i^T \bra{x-c}+ b_i.\]
	 Then $\mybar{Q} = Q \cap C$, where
	 \[ C = \set{\bra{x,t} \in \Realnone \,:\, \max_{\mathcal{I}\subseteq [k]} h_{\mathcal{I}}(x) \le t  }.\]
\end{corollary}
\begin{proof}
Note that for the affine mappings $M, M^{-1}$ given by $M(x) = A\bra{x-c}$ and $M^{-1}(x) = A^{-1}x+c$, we have $Q = M^{-1}\bra{Q_0}$ and $Q_0 = M\bra{Q}$, where $Q_0 = \set{(x,t)\in \Realnone\,:\, \normss{x} \le t}$. Using Lemma \ref{basictransformlemma}, we prove the corollary by finding a closed form expression for $\mybar{Q_0}$ where the forbidden set is a k-branch split $M\bra{F}$ associated with $\tilde{\pi}_i=A^{-T} \pi_i$, $\tilde{\pi}_0^i=\pi_0^i-\pi^T_i c$, and $\tilde{\pi}_1^i=\pi_1^i-\pi^T_i c$ for $i \in [k]$. 
By Lemma \ref{norm decompositioncross}, we have
\beqn
Q_0 = \set{\bra{x,t} \in \Realnone : \norm{P_{\tilde{\Pi}}^{\perp}x}_2^2 +\sum_{i=1}^k\frac{\bra{\tilde{\pi}^T_i x}^2}{\norm{\tilde{\pi}_i}^2_2} \le t.
}
\eeqn
The result then follows from Proposition \ref{GFormSplitCutCross}.
\qed\end{proof}


\section{General intersection cuts through aggregation}\label{intersection_cut_section}

In this section we consider the case in which the base sets are either epigraphs or lower level sets of convex functions and the forbidden sets are hypographs or upper level sets of concave functions. Our cut construction approach in this case is based on a simple aggregation technique, which again can be more naturally explained for epigraphs of specially structured functions. Following the structure of Section~\ref{NSCSec}, we also begin by studying the epigraphical sets and then consider the case of non-epigraphical sets. We end this section by illustrating the power and limitations of the aggregation approach by considering intersection cuts for quadratic sets.

\subsection{Intersection cuts for epigraphs}
Let $G,J:\Real\times \Real \to \Real$ be a convex and a concave function given by $G(z,y)=z^2+2 y^2$ and $J(z,y)=-(z-1)^2+1-y^2$, and let $B = \epi\bra{G} \text{ and } F = \hyp(J)$.
For $\lambda\in[0,1]$, let $W_\lambda(z,y)=(1-\lambda)G+\lambda J$. As illustrated in Figure~\ref{aggcutfig1}, for any $\lambda\in[0,1]$, we have that  $W_\lambda(z,y)\leq t$ is a binding valid cut for $\mybar{B}$. In Figure~\ref{aggcutfig1}, the graph of $G$ is given by the thick black curve, graph of $J$ by the thin blue curve, and valid aggregation cuts $W_\lambda$ for $\lambda\in \{1/4,1/2,3/4\}$ by the red dotted, green dash-dotted, and brown dashed curves, respectively.
Figure~\ref{aggcutfig1} illustrates that, depending on the choice of $\lambda$, the inequality could be non-convex, or it could be convex but not sufficient. It is clear from  the figure that, in this case, the correct choice of $\lambda$ is  $1/2=\arg\max\set{\lambda\in [0,1]\,:\, W_\lambda \text{ is convex}}$, which yields the strongest convex cut from this class. Furthermore, as illustrated in Figure~\ref{aggcutfig2}, we have that for any $\bra{\overline{z},\overline{y},\overline{t}}\in \epi(W_{1/2})\cap \Int\bra{F}$, we can find friends in $\epi\bra{W_{1/2}}\cap\bd\bra{F}$ by following the slope of $W_{1/2}$ similar to what we did in Section~\ref{SepFunSplit} for split cuts of separable functions. We can then show that 
 \begin{equation*}
	 \mybar{B} = B \cap\epi\bra{W_{1/2}}.
\end{equation*}
A similar construction can also be obtained if we instead study $\conv\bra{\set{(z,y,t)\in \epi(G)\,:\, J(z,y)\leq 0}}$.
\begin{figure}[htb]
\centering 
\subfigure[Various aggregations of $G$ and $J$.]{\includegraphics[scale=0.5]{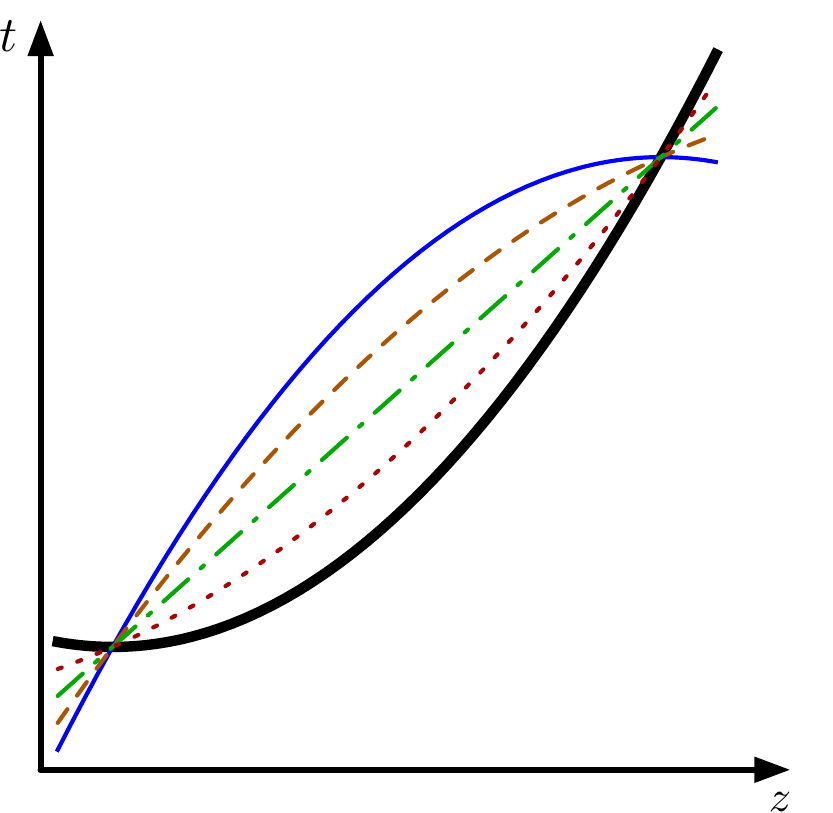}\label{aggcutfig1}}
\subfigure[Friends construction by following slope of $W_{1/2}$.]{\includegraphics[scale=0.5]{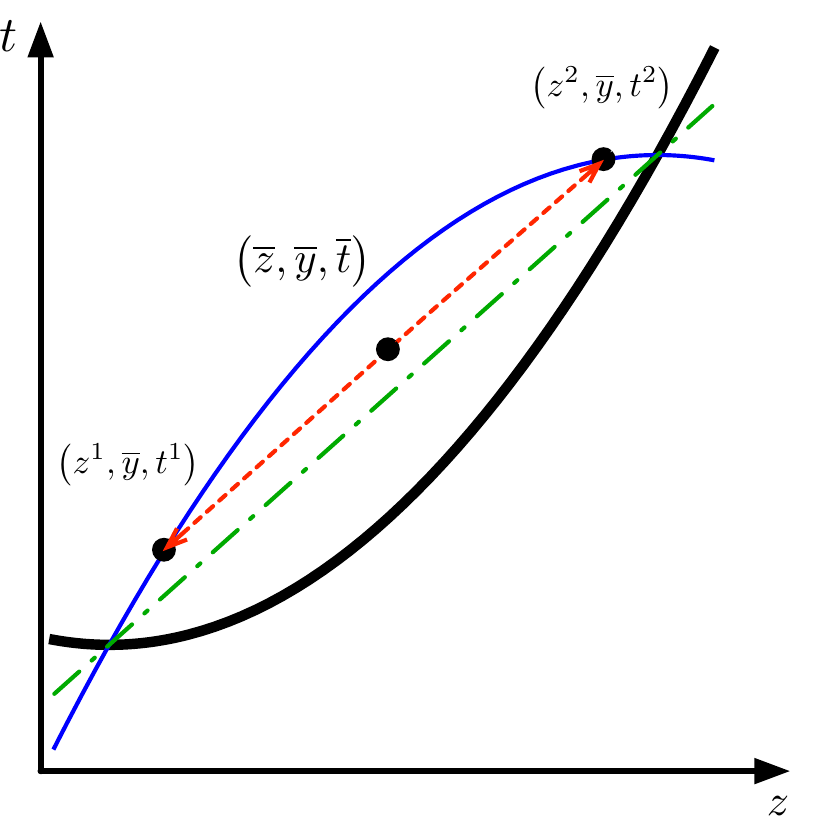}\label{aggcutfig2}}
\caption{Cuts from aggregation.}\label{aggcutfig}
\end{figure}

$W_\lambda$ and the convexity requirement on it are the basis of many techniques such as Lagrangian/SDP relaxations of quadratic programming problems \cite{fujie1997semidefinite,oustry2001sdp,polik2007survey,poljak1995recipe}, the QCR method for integer quadratic programming \cite{billionnet2012extending,billionnet2009improving}, and an algorithm for constructing projected SDP representations of the convex hull of quadratic constraints introduced in  \cite{yildiran2009convex}. It is hence not surprising that the approach works in the quadratic case. However, as shown in \cite{yildiran2009convex}, even in the quadratic case the approach can fail to yield convex constraints or closed form expressions. Furthermore, for general functions, $W_\lambda$ can easily be non-convex for every $\lambda$. Fortunately, as the following proposition shows, the aggregation approach can yield closed form expressions for general intersection cuts for problems with special structures.

 \begin{proposition}\label{generalintprop} Let $g_i:\Real \to\Real$ be convex functions for each $i\in[n]$, $m,h\in\Real^n$, $r,q\in \Real$, and $\gamma\in \Real_+$. Furthermore, let $\set{a_i}_{i=1}^n\subseteq \Real^n$ be such that $a_n\neq 0_n$ and $a_i \perp a_j$ for every $i\neq j$, and $\set{\alpha_i}_{i=1}^n\subseteq \Real_+$ be such that $0\neq \alpha_n\geq \alpha_i$ for all $i$.  Let 
	 \begin{align*}
		 G(x)=&\sum_{i=1}^n g_i\bra{a_i^Tx}+m^Tx+r,\\
		 J(x)=&-\sum_{i=1}^n \alpha_i g_i\bra{a_i^Tx}-h^Tx-q,
	 \end{align*}
	 $B:=\epi(G)$, and $F:=\set{\bra{x,t}\in\Realnone\,:\,  \gamma t\leq J(x)}$.  If $\bra{1+\gamma/\alpha_n}>0$ and

	 \begin{equation}\label{intersectlimcondition}
		 \lim_{|s|\to \infty} -\alpha_n g_n\bra{ s a_n^Ta_n} - s\bra{h^Ta_n+\gamma \frac{\bra{m-h/\alpha_n}^T a_n}{1+\gamma/\alpha_n}}=-\infty,
	 \end{equation}
	 then 
	 \begin{equation}\label{intersect1}
		 \mybar{B}=	 \conv\bra{\set{\bra{x,t}\in \epi(G)\,:\, J(x)\leq \gamma t}}=\epi(G)\cap \epi(W),
	 \end{equation}
	 where
	 \begin{equation}
		 W(x):=\frac{G(x)+(1/\alpha_n)J(x)}{1+\gamma/\alpha_n}=\frac{\sum_{i=1}^{n-1} \bra{1-\alpha_i/\alpha_n} g_i\bra{a_i^Tx}+\bra{m-h/\alpha_n}^Tx+(r-q/\alpha_n)}{\bra{1+\gamma/\alpha_n}}.
	 \end{equation}
 \end{proposition}
 \begin{proof}
 The first equality in \eqref{intersect1} is direct. For the second equality, we proceed as follows. 
	  $W$ is a non-negative linear combination of $G$ and $J$ that is also a convex function from which it is easy to see that  the left to right containment  holds.

	  To show the right to left containment, let $\bra{\overline{x},\overline{t}}\in \epi(G)\cap \epi(W)$ be such that $J\bra{\overline{x}}> \gamma \overline{t}$. Let $k= \frac{\bra{m-h/\alpha_n}^Ta_n}{1+\gamma/\alpha_n}$. Because of \eqref{intersectlimcondition}, there exits $s_1>0$ and $s_2<0$, for which $\bra{x^i,t^i}= \bra{\overline{x}+s_i a_n,\overline{t}+s_i k}$ for $i=1,2$ are such that $J\bra{x^i}= \gamma t^i$. Furthermore, by design, $\bra{x^i,t^i}\in \epi(W)$ for $i=1,2$ which implies $G\bra{x^i}+J\bra{x^i}/\alpha_n\leq \bra{1+\gamma/\alpha_n} t^i$ and hence $G\bra{x^i}\leq  t^i$. The result then follows by noting that $ \bra{\overline{x},\overline{t}}\in\conv\bra{\set{\bra{x^1,t^1},\bra{x^2,t^2}}}$.
 \qed\end{proof}
 \subsection{Intersection cuts for level sets}
 
We can extend the aggregation approach to  certain non-epigraphical sets through the following proposition whose proof is a direct analog to that of Proposition~\ref{generalintprop}. 

 \begin{proposition}\label{boundinterprop} Let $g_i:\Real \to\Real$ be convex functions for each $i\in[n]$, $m \in\Real^n$, $r,q\in \Real$. Furthermore, let $\set{a_i}_{i=1}^n\subseteq \Real^n$ be such that $a_n\neq 0_n$ and $a_i \perp a_j$ for every $i\neq j$, and $\set{\alpha_i}_{i=1}^n\subseteq \Real_+$ be such that $0\neq \alpha_n\geq \alpha_i$ for all $i$.  Let 
	 \begin{align*}
		 G(x)&=\sum_{i=1}^n g_i\bra{a_i^Tx}+m^Tx+r,\\
		 J(x)&=-\sum_{i=1}^n \alpha_i g_i\bra{a_i^Tx}-\alpha_n m^Tx-q,
	 \end{align*}
	 $B:=\set{x\in\Real^n\,:\,  G(x)\leq 0}$, and $F:=\set{x\in\Real^n\,:\,  J(x)\geq 0}$. If
	 \begin{equation}
		 \lim_{|s|\to \infty} -\alpha_n g_n\bra{ s a_n^Ta_n} - s\alpha_n m^Ta_n=-\infty,
	 \end{equation}
	 then 
	 \begin{equation}	\mybar{B}=	 \conv\bra{\set{x\in \Real^n\,:\, \begin{aligned}G(x)&\leq 0,\\J(x)&\leq 0\end{aligned}}}=\set{x\in \Real^n\,:\, \begin{aligned}G(x)&\leq 0,\\W(x)&\leq 0\end{aligned}},
	 \end{equation}
	 where
	 \begin{equation}
		 W(x):=G(x)+(1/\alpha_n)J(x)=\sum_{i=1}^{n-1} \bra{1-\alpha_i/\alpha_n} g_i\bra{a_i^Tx}+(r-q/\alpha_n).
	 \end{equation}
 \end{proposition}

The special structure in both of these propositions is extremely simple, but thanks to the symmetry of the quadratic constraints, they can be used to get formulas for several quadratic intersection cuts. 

\subsection{Intersection cuts for quadratic sets}

\begin{corollary}\label{coro} Let $A\in \Real^{n\times n}$ be an invertible matrix, $D\in \Real^{n\times n}$, $c,d\in \Real^n$, $q \in \Real$, $\gamma \in\Real_+$,  

\[Q := \set{\bra{x,t} \in \Realnone : \|A\bra{x-c}\|_2^2 \leq t},\]
and
	\[F:=\set{\bra{x,t}\in \Realnone\,:\,\gamma t+q\leq -\norm{D\bra{x-d}}^2}.\]
	Then 

\begin{equation}
	\mybar{Q}=\set{\bra{x,t}\in \Realnone\,:\,\begin{aligned}\norm{A\bra{x-c}}^2&\leq t\\x^TNx+a^Tx+f&\leq (\alpha_n+\gamma)t\end{aligned}},
\end{equation}
for
\[N=A^TRA, \]
\[a=-2A^Te-2A^TRAc, \]
\[f=c^TA^TRAc+2\bra{A^Te}^Tc-w-q,\]
\[ R = \sum_{i=1}^{n-1} \bra{\alpha_n-\alpha_i}v_i v_i^T,\]
\[ e = \sum_{i=1}^n\alpha_iv_i^T A (c-d) v_i,\]
\[ w = \sum_{i=1}^n\alpha_i\bra{v_i^T A(c-d)}^2,\]
where $\bra{v_i}_{i=1}^n\subseteq\Real^n$ and $\bra{\alpha_i}_{i=1}^n\subseteq \Real$ correspond to an eigenvalue decomposition of $A^{-T}D^TDA^{-1}$ so that 
\[A^{-T}D^TDA^{-1}=\sum_{i=1}^n\alpha_i v_i v_i^T,\]
$\norms{v_i} =1$ for all $i \in [n]$, $v_i^T v_j = 0$ for all $i \neq j$, and $\alpha_n\geq \alpha_i$ for all $i \in [n]$.
\end{corollary}
\begin{proof}
Let $y=A(x-c)$ and $T \eqq Q\setminus\Int\bra{F}$. Using orthonormality of the vectors $v_i$, $T$ can be written on the $y$ variables as
\[T=\set{\bra{y,t}\in \Realnone\,:\,\begin{aligned}
	\sum_{i=1}^n \bra{v_i^Ty}^2&\leq t\\
	-\sum_{i=1}^n \alpha_i \bra{v_i^Ty}^2-2e^T y-w-q &\leq \gamma t\end{aligned}}.\]
	The result then follows by using Proposition~\ref{generalintprop}. 
\qed\end{proof}


An interesting case of Corollary~\ref{coro} arises when $\gamma=0$. In this case, the base  set $B$ corresponds to a paraboloid and the forbidden set $F$ corresponds to an ellipsoidal cylinder. In such case,  the minimization of $t$ over $(x,t)\in B\setminus\Int\bra{F}$ is equivalent to the minimization of a convex quadratic function outside an ellipsoid, which corresponds to the simplest indefinite version of the well known trust region problem. While this is a non-convex optimization problem, it  can be solved in polynomial time through Lagrangian/SDP approaches \cite{polik2007survey}. It is known that  optimal dual multipliers of an SDP relaxation of a non-convex quadratic programming problem such as the trust region problem can be used to construct a finite convex quadratic optimization problem with the same optimal value as the original non-convex  problem (e.g. \cite{DBLP:conf/ipco/GiandomenicoLRS11}). Furthermore, the complete feasible region induced by an SDP relaxation on the original space (in this case $(x,t)$) can be characterized by  an infinite number of   convex quadratic constraints \cite{kojima2000cones}. This characterization has recently been simplified for the feasible region of the trust region problem in   \cite{dan}. This work gives a semi-infinite characterization of   $T$ for $\gamma=0$ composed by the convex quadratic constraint $\normss{A\bra{x-c}}\leq t$ plus an infinite number of linear inequalities that can be separated in polynomial time. Corollary~\ref{coro} shows that these linear inequalities can be subsumed by a single convex quadratic constraint, which gives another explanation for their polynomial time separability\footnote{After our original submission, it was brought to our attention that reduction of the infinite number of inequalities to a single quadratic inequality can also be directly deduced from the formulas for such linear inequalities given in \cite{dan}.}. We note that the techniques in  \cite{dan} are also adapted to other non-convex optimization problems (both quadratic and non-quadratic). Hence, combining Corollary~\ref{coro} with these techniques could yield valid convex quadratic inequalities for more general non-convex problems.

Another interesting application of Corollary~\ref{coro} for the case $\gamma=0$ is the Shortest Vector Problem (SVP) \cite{MGbook} of the form $\min\set{\normss{Ax}\,:\, x\in \mathbb{Z}^n\setminus\set{0_n}}$. Similar to the Closest Vector Problems (CVP) studied in Section~\ref{simplequasplitcuts}, we can transform this problem to $\min_{\bra{x,t}\in Y\cap\bra{\mathbb{Z}^n\times \Real}} t$ for \[Y=\set{(x,t)\in \Realnone \,:\,\normss{Ax}\leq t,\, x\neq {0_n}},\]
so that we can strengthen the problem by generating valid inequalities for $Y$. Unfortunately, as the following simple lemma shows, traditional split cuts will not add any strength. 

\begin{lemma}
Let $Y_0 = Y \cup \set{{({0_n}, 0)}}$ and $F$ be a split. For any $A \in \Real^{n \times n}$,
\beqn
t^{*} = \min \set{t \,:\, (x, t) \in \cap_{(\pi, \pi_0) \in \Z^n \times \Z} \,\mybar{Y_0}} = 0.
\eeqn
\end{lemma}
\begin{proof}
Note that for all integer splits $(\pi, \pi_0) \in \Z^n \times \Z$, $(\xbar, \tbar) = ({0_n}, 0)$ belongs to one side of the disjunction. Thus, we have $t^{*} \le 0$ and the result follows from non-negativity of the norm.
\qed\end{proof}

However, we can easily construct \emph{near} lattice free ellipsoids centered at $0_n$ that do not contain any point from $\mathbb{Z}^n\setminus\{0_n\}$ in their interior, and use them to get some bound improvement. For instance, in the trivial case of $A = I$, Corollary~\ref{coro} applied to the single \emph{near} lattice free ellipsoid given by the unit ball $\set{x \in \Real^n \,:\, \norms{x} \le 1}$ yields a cut that provides the optimal value $t^{*}= 1$. Similar ellipsoids could be used to  generate strong convex quadratic valid inequalities for non-trivial cases to significantly speed up the solution of SVP problems. Studying the effectiveness of these cuts is left for future research.

We end this section with a brief discussion about the strength and possible extensions of the aggregation technique. For this, we begin by presenting the following corollary of Proposition~\ref{boundinterprop} whose proof is analogous to that of Corollary~\ref{coro}.
\begin{corollary}\label{coro2} Let $A\in \Real^{n\times n}$ be an invertible matrix, $D\in \Real^{n\times n}$, $c\in \Real^n$, $r_1,r_2\in\Real_+$,
	\[E^2 := \set{x\in \mathbb{R}^n : \|A\bra{x-c}\|_2^2 \leq r_1},\]
	and
	\[F:=\set{x\in \Real^n\,:\,\normss{D\bra{x-c}} \le  r_2}.\]
Then
\begin{equation}
	\mybar{E^2}=\set{x \in \Real^n:\,\begin{aligned}\norm{A\bra{x-c}}^2_2&\leq r_1\\x^TNx+a^Tx+f&\leq 0\end{aligned}},
\end{equation}

\[N=A^TRA, \]
\[a=-2A^TRAc, \]
\[f=c^TA^TRAc+r_2/\alpha_n - r_1,\]
\[ R = \sum_{i=1}^{n-1} \bra{1 - \alpha_i/\alpha_n}v_i v_i^T,\]
where $\bra{v_i}_{i=1}^n\subseteq\Real^n$ and $\bra{\alpha_i}_{i=1}^n\subseteq \Real$ correspond to an eigenvalue decomposition of $A^{-T}D^TDA^{-1}$ so that 
\[A^{-T}D^TDA^{-1}=\sum_{i=1}^n\alpha_i v_i v_i^T,\]
$\norms{v_i} =1$ for all $i \in [n]$, $v_i^T v_j = 0$ for all $i \neq j$, and $\alpha_n\geq \alpha_i$ for all $i \in [n]$.
\end{corollary}

Corollary \ref{coro2} shows how to  construct the convex hull of the set obtained by removing an ellipsoid or an ellipsoidal cylinder from an ellipsoid. However, this construction only works if the ellipsoids have a common center $c$. The following example shows how the construction can fail for non-common centers. In addition, the example shows that the aggregation technique does not subsume the interpolation technique and sheds some light into the relationship between Corollaries~\ref{coro} and \ref{coro2} and SDP relaxations for quadratic programming. 
\begin{example} Let $B=\set{\bra{z,y} \in \Real^2\,:\, z^2+y^2\leq 4}$ and $F$ be a split associated with the split disjunction $z\leq 0 \vee z\geq 1$. From Corollary~\ref{ellip coro}, we have that
	\bearn
		\mybar{B}&:=&\conv\bra{\set{\bra{z,y}\in B\,:\, z\leq 0}\cup \set{\bra{z,y}\in B\,:\, z\geq 1}}\\
		&=&\set{\bra{z,y} \in B \,:\, |y|\leq \bra{\sqrt{3}-2}z+2}.
	\eearn
Now let $G(z,y)= z^2+y^2-4$ and $J(z,y)=-(z-1/2)^2+1/4$. Since split disjunction $z\leq 0 \vee z\geq 1$ is equivalent to $J(z,y)\leq 0$, we have $\mybar{B}=\conv\bra{S}$, where
\begin{equation}
	S = \bra{\set{\bra{z,y}\in \Real^2\,:\, G(z,y)\leq 0,\quad J(z,y)\leq 0}}.
\end{equation}
Now consider $W_\lambda=(1-\lambda) G+\lambda J$. One can check that the split cut $|y|\leq \bra{\sqrt{3}-2}z+2$ obtained through Corollary~\ref{ellip coro}, can be equivalently written as 
\begin{subequations}
\begin{align}
	y^2-\bra{\bra{\sqrt{3}-2}z+2}^2&\leq 0 \label{splitcut} \\
	\label{boundconst}\bra{\sqrt{3}-2}z+2&\geq 0.
\end{align}
\end{subequations}
In turn, \eqref{splitcut} is equivalent to $W_{\lambda^*}\leq 0$ for $\lambda^*=\frac{4}{33}\bra{6-\sqrt{3}}$ because $W_{\lambda^*} / \bra{\frac{1}{33} \bra{9+4\sqrt{3}}} = y^2-\bra{\bra{\sqrt{3}-2}z+2}^2$. By noting that \eqref{boundconst} holds for $B$, we conclude that 
\begin{equation}
	\mybar{B} =\set{\bra{z,y}\in B\,:\, W_{\lambda^*}(z,y)\leq 0}.
\end{equation}
Unfortunately, $W_{\lambda^*}$ is not a convex function, so it does not fit in the aggregation framework described in this section.  In particular, $W_{\lambda^*}$ is an indefinite quadratic function so it cannot be obtained from an SDP relaxation of $S$. 
Indeed, we can show that the SDP relaxation of $S$ strictly contains $\mybar{B}$. Finally, while we can obtain $W_{\lambda^*}$ through a procedure described in \cite{yildiran2009convex}, this procedure requires the execution of a numerical algorithm and does not give closed form expressions such as those provided by Corollary~\ref{ellip coro}.
\end{example}

\section{Final remarks and future work}

We introduced two techniques that can be used to construct formulas for split, k-branch split, and general intersection cuts for several classes of convex sets. While obtaining closed form expressions of these formulas requires sets with specific structures, the techniques can yield general intersection cuts for a wide range of non-polyhedral sets including quadratic sets.  Furthermore, the  independence of the approaches on the specific class of the considered convex set (e.g. quadratic, semi-algebraic, etc.) suggests a high potential for extensibility to other settings by perhaps sacrificing  closed form expressions in favor of numerical methods. For instance, consider the approach described in Section~\ref{tincsec}. While this approach was used  in Sections~\ref{simplequasplitcuts} and \ref{complicated}  to obtain closed form expressions of split cuts for quadratic sets, it may not be successful when applied to sets that are not semi-algebraic or quadratic. However, the approach may be successful in numerically constructing split cuts for a given disjunction (i.e., when $\pi, \phat, \pi_0$, and $\pi_1$ are fixed to certain numerical values).

With regards to the potential effectiveness of the developed cuts in the context of solution methods for MINLP, we note that adding such nonlinear cuts to the continuous relaxation of a MINLP could significantly increase its solution time. Hence there will likely be a strong trade-off between the strength provided by such cuts and their computational cost. It is then unclear if such nonlinear cuts can provide a significant computational advantage over linearization approaches such as those in \cite{DBLP:conf/ipco/Bonami11,mustafa}, which do not require explicit cut formulas. However, even in such cases, the developed nonlinear cuts can provide valuable information about the performance of the linearization approaches. For instance, the linearization approaches can sometimes require a large number of iterations to yield a bound improvement similar to that obtained with the associated nonlinear cut. Adding the nonlinear cut provides a simple way to evaluate if the lack of bound improvement is due to lack of strength of the cut  or lack of convergence of the linearization approach. 
Similarly, the availability of explicit formulas of split cuts for quadratic sets proven extremely useful to evaluate the strength of a cutting plane approach based on extended formulations  in \cite{orlpaper}. We are further exploring the computational effectiveness of the interpolation and aggregation techniques and the techniques in \cite{orlpaper}.

\section{Appendix}\label{appendix}
Here we provide the omitted proofs and auxiliary lemmas.


\ghyperlemma*
\begin{proof}
We show the equivalent version of the lemma given by
\begin{enumerate}[(i)]
  \item If $s \in \set{\pi_0, \pi_1}$, then $\abs{as +b} =(\abs{s}^p + \abs{l}^p)^{1/p}$ and
  \item if $s \notin \bra{\pi_0, \pi_1}$, then $\abs{as +b} \le (\abs{s}^p + \abs{l}^p)^{1/p}$.

\end{enumerate}

Let $f(s):=a s +b$ and $g(s):=(\abs{s}^p + \abs{l}^p)^{1/p}$. By definition of $a$ and $b$ we have that $f( \pi_i)=g(\pi_i)$ for $i\in\set{0,1}$. Indeed, $f(s)$ is the (affine) linear interpolation of $g(s)$ through $z=\pi_0$ and $z = \pi_1$. Convexity of $g(s)$ then implies $f(s)\le g(s)$ for all $s\notin \bra{\pi_0, \pi_1}$. If $\abs{\pi_0}=\abs{\pi_1}$, then $\abs{as +b} =f(s)$ and the result follows directly. If $\abs{\pi_0}\neq\abs{\pi_1}$, one can check that $\abs{as +b} =f(s)$ for $s\in \left[\pi_0, \pi_1\right]$ and hence (i) holds. For (ii) it suffices to show that $-as-b\leq g(s)$ for all $s\in \Real$. To show this we first assume $a>0$ and hence $\pi_1>0$ (case $a<0$ is analogous). Because $f(s)$ is affine and $f( \pi_i)=g(\pi_i)$ for $i\in\set{0,1}$, by a sub-differential version of the mean value theorem we have that there exists $\bar{s}\in (\pi_0,\pi_1)$ such that $a \in \partial g(\bar{s})$. Then, by symmetry of $g(s)$ and its convexity, we have that $g(s)\geq g(-\bar{s})-a(s+\bar{s}) = - as + g(-\bar{s})-a\bar{s}$ for $s \in \Real$. The result then follows by noting that $ g(-\bar{s})-a\bar{s}\ge-b$ for all $\bar{s}\in (\pi_0,\pi_1)$ because $g(s)-as \ge 0$ for all $s \in \Real$ and $-b\le0$.
\qed
\end{proof}
%

\cylinderlemma*
 \begin{proof}
We first prove the second case $\pi\notin L^\perp$. The left to right containment follows from $B \setminus \Int\bra{F} \subseteq B$ and convexity of $B$. To show the right to left containment, let $\xbar \in B$ such that $\pi^T \xbar \in \bra{\pi_0,\pi_1}$ and $u \in L$. Note that $\pi\notin L^\perp$ implies $\pi^T u \neq 0$. Let   $x^i \eqq \xbar + \lambda_i u$ for $i \in \set{0,1}$, where 
\beqn
\lambda_i = \frac{\pi_i - \pi^T \xbar}{\pi^T u},
\eeqn
and let  $\beta \in (0,1)$ be such that $\pi^T \xbar = \beta \pi_0 + \bra{1-\beta} \pi_1$. Because $u \in L$ and since $\pi^T x^i = \pi_i$,  we have $x^i \in B \setminus \Int\bra{F}$ for $i \in \set{0,1}$. The results then follows by noting that $\xbar = \beta x^0 + \bra{1-\beta} x^1$.

We prove the first case by showing that
\bear
\conv\bra{B \setminus \Int\bra{F}} &=& \conv\bra{\bra{B_0 + L} \setminus \Int\bra{F}} \label{tp1} \\
&=& \conv\bra{B_0 \setminus \Int\bra{F}} + L \label{tp3} \\
&=& \bra{B_0 \cap C} + L \label{tp4} 
\eear
Note that \eqref{tp1} and \eqref{tp4} follow from the assumptions.
To show the left to right containment in \eqref{tp3}, let $\xbar \in \conv\bra{\bra{B_0 + L} \setminus \Int\bra{F}}$. There exist $y^i \in B_0$, $u^i \in L$ for $i \in \set{0,1}$, and $\beta \in [0,1]$ such that for $x^i \eqq y^i + u^i$, we have $x^i \notin \Int\bra{F}$ and $\xbar = \beta x^0 + \bra{1-\beta} x^1$. Note that $\pi\in L^\perp$ and $x^i \notin \Int\bra{F}$ imply $y^i \notin \Int\bra{F}$ for $i \in \set{0,1}$. The result then follows from noting that $\beta y^0 + \bra{1-\beta} y^1 \in \conv\bra{B_0 \setminus \Int\bra{F}}$ and $\beta u^0 + \bra{1-\beta} u^1 \in L$.

To show the right to left containment in \eqref{tp3}, let $\xbar \in \conv\bra{B_0 \setminus \Int\bra{F}} + L$. There exist $u \in L$, $y^i \in B_0 \setminus \Int\bra{F}$ for $i \in \set{0,1}$, and $\beta \in [0,1]$ such that $\xbar = \beta y^0 + \bra{1-\beta} y^1 + u$. If $\beta \in \set{0,1}$, the result follows by noting that $\pi\in L^\perp$ and $y^0,y^1 \notin \Int\bra{F}$ imply $\xbar \notin \Int\bra{F}$. Assume $\beta \in \bra{0,1}$ and let $x^0 \eqq y^0 + \frac{u}{2\beta}$ and $x^1 \eqq y^1 + \frac{u}{2\bra{1-\beta}}$. The result then follows by noting that $x^i \in B_0+L \setminus \Int\bra{F}$ for $i \in \set{0,1}$ and $\xbar = \beta x^0 + \bra{1-\beta} x^1$.
\qed
\end{proof}

\tparaboloidprop*
\begin{proof}

We first prove the last case where $\phat > 0$ and $\frac{-\normss{\pi}}{4\phat} \le \pi_0 < \pi_1$, or $\phat < 0$ and $\pi_0 < \pi_1 \le \frac{-\normss{\pi}}{4\phat}$ using Proposition~\ref{method1Cross}. Using Lemma \ref{norm decomposition} we have
\beq
C  = \set{ (x,t) \in \Realnone \,:\, \normss{P_\pi^\perp x} \le \bra{c \pi^T x + d t + e}^2 - \frac{\bra{\pi^T x + b}^2}{\normss{\pi}}, \quad c \pi^T x + d t + e \ge 0} \label{qineq}.
\eeq
Now consider the following two cases.

\noindent\textbf{Case 1.} Assume that $\norms{\pi} \ne 0$. To prove the right to left containment in \eqref{generalintcondbCross}, let $(\xbar, \tbar) \in C \cap \bd\bra{S}$. We need to show that
\beq
\bra{c \pi^T \xbar + d\tbar + e}^2 - \frac{\bra{\pi^T \xbar + b}^2}{\normss{\pi}} = \tbar - \frac{\bra{\pi^T \xbar}^2}{\normss{\pi}}. \label{15a_para}
\eeq
Replacing $\tbar$ with $\bra{\pi_i - \pi^T \xbar}/\phat$ for $i \in \set{0,1}$, one can check that \eqref{15a_para} follows from the definition of $b, c, d,$ and e.
To prove the left to right containment in \eqref{generalintcondbCross}, let $(\xbar, \tbar) \in Q_0 \cap \bd\bra{S}$. We only need to show that $c \pi^T \xbar + d\tbar + e \ge 0$. Since $d = c \phat$, we need to show that $c \bra{\pi^T \xbar + \phat \tbar} \ge -e$, which after a few simplifications, can be written as
\beq
\phat \bra{\pi^T \xbar + \phat \tbar} \ge - \bra{\normss{\pi} + \sqrt{\normss{\pi}+ 4 \pi_0 \phat} \sqrt{\normss{\pi}+ 4 \pi_1 \phat}}/{4}. \label{nonneg}
\eeq
\eqref{nonneg} follows from noting that $\min \set{\phat \bra{\pi^Tx + \phat t} \,:\, (x, t) \in Q_0} = -\frac{\normss{\pi}}{4}$.

To show \eqref{generalintcondcCross}, let $\bra{\xbar,\tbar} \in Q_0 \setminus \Int\bra{S}$. Proving $c \pi^T \xbar + d\tbar + e \ge 0$ is similar as in case 1. We only need to show that $(\xbar, \tbar)$ satisfies the quadratic inequality in \eqref{qineq}, which we prove by showing that
\beq
\bra{\bra{c \pi^T \xbar + d\tbar + e}^2 - \frac{\bra{\pi^T \xbar + b}^2}{\normss{\pi}}} - \bra{\tbar - \frac{\bra{\pi^T \xbar}^2}{\normss{\pi}}} \ge 0. \label{rhs}
\eeq
One can check that proving \eqref{rhs} is equivalent to showing that
\beqn
\frac{f^2 \bra{\pi^T \xbar + \phat \tbar - \pi_0} \bra{\pi^T \xbar + \phat \tbar - \pi_1}}{2 \bra{\pi_1 - \pi_0}^2 \phat^2} \ge 0,
\eeqn
which follows from $\pi^T \xbar + \phat \tbar \notin (\pi_0, \pi_1)$. Note that $C$ is a conic set with apex $\bra{x^*,t^*} = \bra{\frac{-b}{\normss{\pi}} \pi,\frac{bc-e}{d}}$. Furthermore, 
\beqn
\bra{\pi,\phat}^T \bra{x^*,t^*} = -e/c = \frac{-\normss{\pi}}{4\phat} - \frac{\sqrt{\normss{\pi} + 4 \pi_0 \hat{\pi}} \sqrt{\normss{\pi} + 4 \pi_1 \hat{\pi}}}{4\phat}.
\eeqn
Hence, if $\phat < 0$, then $\bra{\pi,\phat}^T \bra{x^*,t^*} \ge \frac{-\normss{\pi}}{4\phat} \ge \pi_1$ and if $\phat > 0$, then $\bra{\pi,\phat}^T \bra{x^*,t^*} \le \frac{-\normss{\pi}}{4\phat} \le \pi_0$. Friends condition \eqref{friendcond} then follows from Proposition~\ref{freefriendsconeprop}.

\noindent\textbf{Case 2.} If  $\norms{\pi} = 0$, $C$ is simplified to 
\beq
C = \set{ (x,t) \in \Realnone \,:\, \normss{x} \le \bra{dt+e}^2, \quad dt+e \ge 0}.
\eeq

Interpolation condition \eqref{generalintcondbCross} follows from noting that $\bra{d \tbar + b}^2 = \tbar$. Non-negativity of $d, e$, and $t$ also imply $d \tbar + e \ge 0$. Proving  \eqref{generalintcondcCross} is equivalent to showing that 
\beqn
\frac{f^2 \bra{\phat \tbar - \pi_0} \bra{\phat \tbar - \pi_1}}{2 \bra{\pi_1 - \pi_0}^2 \phat^2} \ge 0,
\eeqn
which follows from $\phat \tbar \notin (\pi_0, \pi_1)$. Note that $C$ is a conic set with apex $\bra{x^*,t^*} = \bra{0,\frac{-e}{d}}$. Furthermore, 
\beqn
\bra{\pi,\phat}^T \bra{x^*,t^*} = -e/c.
\eeqn
As shown in Case~1, we have $\bra{\pi,\phat}^T \bra{x^*,t^*} \notin \bra{\pi_0,\pi_1}$. Friends condition \eqref{friendcond} then follows from Proposition~\ref{freefriendsconeprop}.

To prove the other cases, let $S_0 \eqq \set{(x,t) \in Q_0\,:\, \pi^Tx+\phat t \le \pi_0}$ and $S_1 \eqq \{(x,t) \in Q_0\,:\, \pi^Tx+\phat t \ge \pi_1\}$. Consider the first case where $\phat > 0$ and $\pi_1 \le \frac{-\normss{\pi}}{4\phat}$. We prove the result by showing that $S_0 = \emptyset$ and $S_1 = Q_0$. If $\norm{\pi}_2 = 0$, the result follows from non-negativity of $t$. Now assume that $\norm{\pi}_2 \neq 0$. Note that if $S_0 \neq \emptyset$, one can find $(\xbar, \tbar) \in S_0$ such that $\bra{\pi^T \xbar}^2/\norm{\pi}_2^2 \le \bra{\pi_0 - \pi^T \xbar}/\hat{\pi}$. Therefore, we prove $S_0 = \emptyset$ by showing that $\bra{\pi^T x}^2/\norm{\pi}_2^2 >  \bra{\pi_0 - \pi^T x}/\hat{\pi}$. This follows from noting that for $y \in \Real$, the quadratic equation $\frac{y^2}{\norm{\pi}_2^2} =  \frac{\pi_0 - y}{\hat{\pi}}$ does not have any solution.
To prove $S_1 = Q_0$, we show that $\pi^T x + \phat t \ge \pi_1$ is a valid inequality for $Q_0$. This comes from the fact that the quadratic equation $\frac{y^2}{\norm{\pi}_2^2} =  \frac{\pi_1 - y}{\hat{\pi}}$ has at most a single solution and as a result, we have $\bra{\pi_1 - \pi^T x}/\hat{\pi} \le \bra{\pi^T x}^2/\norm{\pi}_2^2 \le t$. The proof for the case $\phat < 0$ and $\frac{-\normss{\pi}}{4\phat} \le \pi_0$ is analogous and follows by noting that $S_0 = Q_0$ and $S_1 = \emptyset$.

Finally, the second case $\phat > 0$ and $\pi_0 < \frac{-\normss{\pi}}{4\phat} < \pi_1$. We prove the result by showing that $S_0 = \emptyset, S_1 \subsetneq Q_0$, and $S_1 \neq \emptyset$. Proving $S_0 = \emptyset$ is analogous to the previous case. We have $S_1 \subsetneq Q_0$ since $\bra{\bar{x}, \tbar} = \bra{\frac{- \pi}{2 \phat}, \frac{\normss{\pi}}{4 \phat^2}} \in Q_0$, but $\bra{\bar{x}, \tbar} \notin S_1$. To prove $S_1 \neq \emptyset$, one can check that for any $\bar{x} \in \Real^n$ and $\tbar = \Max \set{\normss{\xbar}, \frac{\pi_1 - \pi^T \bar{x}}{\hat{\pi}}}$, $\bra{\bar{x}, \tbar} \in S_1$. The proof for third case $\phat < 0$ and $\pi_0 < \frac{-\normss{\pi}}{4\phat} < \pi_1$ is analogous and follows by noting that $S_1 = \emptyset, S_0 \subsetneq Q_0$, and $S_0 \neq \emptyset$.
\qed
\end{proof}

\tconeprop*
\begin{proof}
We first prove the last case $0 \in (\pi_0, \pi_1)$ and $\phat \in (-\norm{\pi}, \norm{\pi})$ using Proposition~\ref{method1Cross}. Note that $\phat \neq 0$ and $\phat \in (-\norm{\pi}, \norm{\pi})$ imply $\norms{\pi} \neq0$. Using Lemma \ref{norm decomposition} we have 
\beq
C  = \set{ (x,t) \in \Realnone \,:\, \normss{P_\pi^\perp x} \le \bra{c \pi^T x + d t + e}^2 - \frac{\bra{a\pi^T x + b}^2}{\normss{\pi}}, \quad c \pi^T x + d t + e \ge 0} \label{qineq2}.
\eeq
Note that $d>0$. Similarly to the proof of Proposition~\ref{newpnlscut prop}, one can show that interpolation condition \eqref{generalintcondCross} holds by the definition of $a, b, c, d$, and $e$. If $\abs{\pi_0 } = \abs{\pi_1 }$, then $u = \bra{\pi,\frac{-c \normss{\pi}}{d}} \in \lin\bra{C}$ and friends condition \eqref{friendcond} follows from Proposition~\ref{freefriendslinprop}. If $\abs{\pi_0 } \ne \abs{\pi_1 }$, then $C$ is a conic set with apex $\bra{x^*,t^*} = \bra{\frac{-b}{a \normss{\pi}} \pi,\frac{bc-ae}{ad}}$. Furthermore, 
\beqn
\bra{\pi,\phat}^T \bra{x^*,t^*} = \frac{2 \pi_0\pi_1}{\pi_0 + \pi_1}.
\eeqn
If $\pi_0 + \pi_1 <0$, then one can check that $\frac{2 \pi_0\pi_1}{\pi_0 + \pi_1} \ge \pi_1$, and if $\pi_0 + \pi_1 >0$, then one can check that $\frac{2 \pi_0\pi_1}{\pi_0 + \pi_1} \le \pi_0$. Friends condition \eqref{friendcond} then follows from Proposition~\ref{freefriendsconeprop}.

To prove the first case $0 \notin \bra{\pi_0, \pi_1}$, we only need to show that friends condition \eqref{friendcond} holds. This follows from Proposition~\ref{freefriendsconeprop} by noting that $K_0$ is a conic set whose apex is the origin.

Finally, we prove the second and third cases. Let $S_0 \eqq \set{(x,t) \in K_0\,:\, \pi^Tx+\phat t \le \pi_0}$ and $S_1 \eqq \{(x,t) \in K_0\,:\, \pi^Tx+\phat t \ge \pi_1\}$. Consider the second case $0 \in \bra{\pi_0,\pi_1}$ and $\hat{\pi} \le -\norm{\pi}_2$. We prove the result by showing that $S_1 = \emptyset$, $S_0 \subsetneq K_0$, and $S_0 \neq \emptyset$. If $\norm{\pi}_2 = 0$, the result follows from non-negativity of $t$. Now assume that $\norm{\pi}_2 \neq 0$. Note that if $S_1 \neq \emptyset$, one can find $(\xbar, \tbar) \in S_1$ such that $\bra{\pi^T \xbar}^2/\norm{\pi}_2^2 \le  \bra{\pi_1 - \pi^T \xbar}^2/\hat{\pi}^2$. Therefore, we prove $S_1 = \emptyset$ by showing that $\bra{\pi^T x}^2/\norm{\pi}_2^2 >  \bra{\pi_1 - \pi^T x}^2/\hat{\pi}^2$.
Note that non-negativity of $t$, $\phat < 0$, and $\pi^T x + \hat{\pi} t \ge \pi_1$ imply $\pi^T x \ge \pi_1 > 0$. One can see that $-\pi^T x < \pi_1 - \pi^T x < \pi^T x$, where the first inequality comes from the fact that $\pi_1 > 0$, and the second inequality follows from $\pi_1 \le \pi^T x$ and $- \pi^T x < 0$. Thus, $\bra{\pi^T x}^2 > \bra{\pi_1 - \pi^T x}^2$ and the result follows by noting that $\frac{1}{\norm{\pi}_2^2} \ge \frac{1}{\hat{\pi}^2}$. We have $S_0 \subsetneq K_0$ since $\bra{\bar{x}, \tbar} = \bra{{0_n}, 0} \in K_0$, but $\bra{\bar{x}, \tbar} \notin S_0$. To prove $S_0 \neq \emptyset$, one can check that for any $\bar{x} \in \Real^n$ and $\tbar = \Max \set{\norms{\xbar}, \frac{\pi_0 - \pi^T \bar{x}}{\hat{\pi}}}$, $\bra{\bar{x}, \tbar} \in S_0$. The proof for the third case $0 \in \bra{\pi_0,\pi_1}$ and $\hat{\pi} \ge \norm{\pi}_2$ is analogous and follows by noting that $S_0 = \emptyset$, $S_1 \subsetneq K_0$, and $S_1 \neq \emptyset$.
\qed
\end{proof}

\bibliographystyle{amsplain}
\bibliography{references}

\end{document}